\documentclass[12pt]{article}
\usepackage{epsfig}
\usepackage{authblk}
\usepackage{epstopdf}
\usepackage{graphics}
\usepackage{amsmath}
\usepackage{amsthm}
\usepackage{amssymb}
\usepackage{array}
\usepackage{subfig} 
\usepackage[font={footnotesize}]{caption}
\usepackage{relsize}
\usepackage[nottoc]{tocbibind}
\usepackage{hyperref}
\usepackage{float}

\usepackage{chngcntr}
\usepackage{wrapfig}

\counterwithin*{equation}{section}
  
\newtheorem{thm}{Theorem}[section] 
\newtheorem{lem}[thm]{Lemma}  
\newtheorem{cor}[thm]{Corollary} 
\newtheorem{prop}[thm]{Proposition}  

\newtheorem{hyp}[thm]{Hypothesis}
\newtheorem{claim}[thm]{Claim}
\newtheorem{defn}[thm]{Definition}

\theoremstyle{definition}
\newtheorem{rmk}{Remark}

\title{Rotating Wave Solutions to Lattice Dynamical Systems II: Persistence Results}
\author{Jason J. Bramburger\\ Division of Applied Mathematics\\ Brown University \\ Providence, Rhode Island 02906\\ USA}
\date{} %Comment to include today's date

\begin{document} %Beginning of Document /////////////////////////////////////////////////////////////////////////////////////////////////////////////////////////////////////////////////////////////////////////////////////////////////////////////////////////////////

\maketitle

\begin{center}
{\em This is a post-peer-review, pre-copyedit version of an article published in the Journal of Dynamics and Differential Equations}
\end{center}

\abstract{This work comes as the second part in a series of investigations into the dynamics of rotating waves as solutions to lattice dynamical systems. Such nonlinear waves as solutions to mathematical equations are of great interest throughout the physical sciences due to their association with many electrophysiological pathologies and this investigation aims to further the understanding of rotating waves from a mathematical perspective. Here we focus on so-called Lambda-Omega differential equations, a well-studied generalization of the celebrated Ginzburg-Landau equation, to show that there exists an interval of sufficiently small coupling values for which a rotating wave solution persists. This result is achieved using a wide range of functional analytic tools, primarily in an effort to apply a non-standard Implicit Function Theorem. This work initiates subsequent studies into the dynamics and bifurcations of rotating waves away from the reaction-diffusion equation setting to differential equations for which traditional symmetry based centre manifold reductions cannot be applied.}

\section{Introduction} \label{sec:Introduction} %Section: Introduction ---------------------------------------------------------------------------------------------------------------------

Non-equilibrium dynamics characterize much of the world around us, from the movement of microscopic bacteria to the motion of the planets. Such dynamics are also a central focus in modern analysis of dynamical systems and bifurcation theory. In particular, wave propagation through a given spatial medium remains an intense area of study to modern experts in partial differential equations (PDEs) and the physical sciences. Linearly propagating solutions to differential equations often arise as traveling waves or pulses, whose temporal evolution can be described by a linear translation in space. Solutions of this type often arise quite naturally when considering invading species in spatial ecology or electrical impulses propagating through a one-dimensional medium such as nerves. On the other hand, differential equations may also exhibit rotationally propagating solutions, referred to as rotating waves, which come as time-periodic solutions whose temporal evolution is equivalent to a rotation in space.

Examples of rotating waves are easily found throughout the physical sciences and have been the focal point of many mathematical investigations for decades now. One of the most visually striking examples of a rotating wave is that of spiral waves, which even non-specialists are familiar with from their occurrences in the form of hurricanes and the shape of galaxies. Rigorous scientific investigations of spiral waves date back at least to the work of Winfree where they were found to arise in the form of chemical concentration patterns in a light-sensitive chemical reaction $\cite{Winfree,Winfree2}$. Following Winfree's work there has been significant interest in spiral waves which has lead to the understanding that they can be associated with electrophysiological pathologies. This includes, but is not limited to, cortical spreading depression, hallucinations and ventricular fibrillation $\cite{Beaumont,Gorelova,Huang,Cardiac,KeenerSneyd,Cortical}$. From a mathematical perspective spiral waves (and more generally rotating waves) remain an area with seemingly endless open problems whose solutions can greatly enhance the scientific communities understanding of many biophysical phenomena. 

It is now well known that much of the dynamics and bifurcations of rotating waves arising as solutions to PDEs which exhibit continuous symmetries can be understood through the action of these symmetry groups $\cite{Ashwin,Victor,SSW,SSW2,SSW3}$.  The analysis undertaken in this work comes as the second part of a series of investigations into the dynamics of rotating waves in systems which lack the traditionally exploited symmetries in the PDE context. That is, this work aims to expand the current knowledge of spiral waves from a mathematical perspective by approaching the problem not from a PDE context where space is generally continuous, but by moving to a spatially discrete framework. We attempt to initiate further investigations into the dynamics and bifurcations of rotating waves in this discrete spatial setting in order to compare and contrast with rotating waves in the continuous spatial setting. Thus, we attempt to broaden the scientific communities understanding of nonlinear waves in the absence of key symmetries in the differential equation.             

As previously mentioned, this work builds upon a previous work where much of the investigation here is motivated $\cite{MyWork}$. The reader is urged to begin with this initial investigation since much of the motivation for the specific problem studied here is contained in this previous work, as well as an in-depth understanding of the full problem at hand. In this work we extend the results of this preceding work to obtain a proof of the existence of rotating waves in an infinite system of coupled ordinary differential equations arising from a spatial discretization of reaction-diffusion PDEs. We dedicate the following subsection to precisely describing the system of interest throughout this work.

\subsection{Spatially Discretized Lambda-Omega Systems} %Subsection: My Model ---------------------------------------------------------------------------------------------------------------------

Our work is initiated by considering so-called Lambda-Omega reaction-diffusion equations, typically written in terms of a single complex variable $z(x,y,t):\mathbb{R}^2 \times \mathbb{R} \to \mathbb{C}$, of the form 
\begin{equation} \label{RDELambdaOmega} %Lambda-Omega RDE/PDE
	\frac{\partial z}{\partial t} = D\bigg(\frac{\partial^2 z}{\partial x^2} + \frac{\partial^2 z}{\partial y^2}\bigg) + z[\lambda(|z|) + {\rm i}\omega(|z|)],	
\end{equation}
where ${\rm i}=\sqrt{-1}$ is the imaginary constant. The specific forms of the functions $\lambda$ and $\omega$ remain primarily problem-dependent but typically are taken as generalizations of the functions $\lambda(R) = \pm a^2 \mp R^2$ for some $a > 0$ and $\omega(R)$ a constant function. Since their inception by Howard and Kopell, Lambda-Omega systems have become an archetype for oscillatory behaviour in reaction-diffusion systems $\cite{Kopell}$. These reaction-diffusion equations come as generalizations of the complex Ginzburg-Landau equation which was the central focus of the preceding investigation, and are well-known to arise as the lowest order perturbation of any reaction-diffusion system near a Hopf bifurcation $\cite{Cohen}$. Most importantly is that PDEs of this type are well-known to exhibit rotating wave solutions $\cite{Cohen,Greenberg,Kopell2,Troy}$, and therefore provides a natural starting point for the mathematical investigation presented here. 

Recently spatially discretized partial differential equations, termed lattice dynamical systems (LDSs), have become a central focus for investigations into nonlinear phenomena in the absence of spatial homogeneity. The simplest way to obtain an LDS analogous to $(\ref{RDELambdaOmega})$ would be to introduce a spatial step size $h > 0$ and use the approximation
\begin{equation}
	\frac{\partial^2 u}{\partial x^2}(x,y) \approx \frac{u(x+h,y) + u(x-h,y) - 2u(x,y)}{h^2},
\end{equation} 
along with a similar approximation for $\partial^2u/\partial y^2$. This allows one to move from the continuous spatial medium of $(\ref{RDELambdaOmega})$ to a spatial grid $x = i$ and $y = j$ for $i,j\in h\mathbb{Z}$. In this way we have replaced the second order differential operators in $(\ref{RDELambdaOmega})$ with the following discrete coupling term:
\begin{equation}
	\sum_{i',j'} (z_{i',j'}(t) - z_{i,j}(t)) := z_{i+1,j}(t) + z_{i-1,j}(t) + z_{i,j+1}(t) + z_{i,j-1}(t) - 4z_{i,j}(t),
\end{equation}
where we have written $z(ih,jh,t) = z_{i,j}(t)$ for all $(i,j)\in\mathbb{Z}^2$ to emphasize that our system now is an ordinary differential equation. This brings us to the infinite system of coupled ordinary differential equations which will form the basis of our investigation in this work given by
\begin{equation} \label{LambdaOmegaLDS} %Main Lambda-Omega Equation
	\dot{z}_{i,j} = \alpha\sum_{i',j'} (z_{i'j'} -z_{i,j}) + z_{i,j}[\lambda(|z_{i,j}|) + {\rm i}\omega(|z_{i,j}|,\alpha)], \ \ \ (i,j) \in \mathbb{Z}^2,	
\end{equation} 
where we have suppressed the dependence on $t$ for simplicity. The coupling interactions of $(\ref{LambdaOmegaLDS})$ can be visualized using Figure $\ref{fig:Lattice}$ where the boxes represent the distinct elements and the connecting lines represent the coupling interactions. 

\begin{figure} %Figure: Lattice Visualization Figure
	\centering
	\includegraphics[height=7cm, width=7cm]{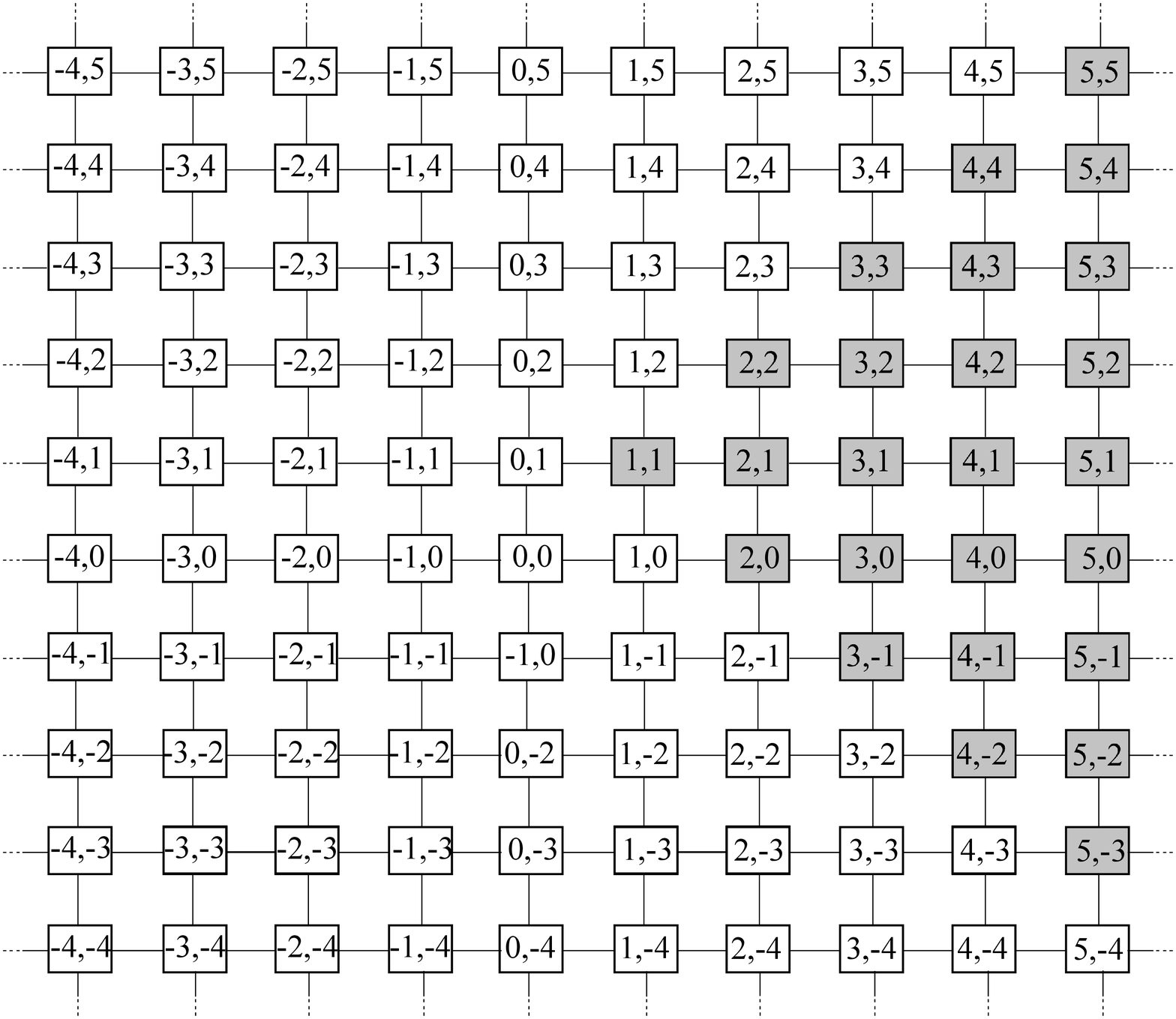}
	\caption{A $2$-dimensional lattice. Connections between cells represent the nearest-neighbour coupling inferred by equation ($\ref{LambdaOmegaLDS}$). The shaded cells represent the indices of the reduced system, as will be defined in Definition $\ref{defn:ReducedSystem}$.}
	\label{fig:Lattice}
\end{figure} 

Here the parameter $\alpha$ is often referred to as the {\em coupling coefficient} and it describes the strength of interaction between neighbouring elements in the lattice. Using the spatial discretization technique described above one can see that $\alpha \approx D/h^2 \geq 0$. The limit $\alpha \to \infty$ corresponds to a return to the continuum equation $(\ref{RDELambdaOmega})$, whereas $\alpha \to 0^+$ is often termed the {\em anti-continuum limit}. It is this anti-continuum limit where we have focused our attention in the previous investigation and we intend to continue this investigation here. In particular, the goal of this work is to prove the existence of rotating wave solutions to system $(\ref{LambdaOmegaLDS})$ for $\alpha > 0$ and small. This work attempts to provide the crucial first step in the investigation of rotating waves in the discrete spatial setting, much like the work of Cohen et. al. in the continuous spatial setting $\cite{Cohen}$, or what the work of Zinner has done for the theory of linearly propagating solutions in the LDS setting $\cite{Zinner}$.  

Aside from being spatial discretizations of PDEs, LDSs have proven themselves extremely useful to model a multitude of physical situations irrespective of their continuous spatial counterparts. Lattice models have been shown to be well-suited in such areas as chemical reaction theory $\cite{Erneux,Erneux2}$, quantum mechanics $\cite{Temam}$, models of neural networks $\cite{Ermentrout,ErmentroutKopell2}$, optics $\cite{Firth}$ and material science \cite{Cahn2,Cook}. Although continuum models could potentially be used in such situations, it appears that the discrete nature of the lattice models is better suited to reflect the discrete nature of the physical settings in each of the aforementioned areas.

\subsection{The Model Assumptions} 

As mentioned in the previous subsection, the specific assumptions on the $\lambda$ and $\omega$ functions vary from study to study. We attempt to study system $(\ref{LambdaOmegaLDS})$ in the broadest generality with regard to our assumptions on the $\lambda$ and $\omega$ functions, therefore throughout this work we will make the following assumption: 

\begin{hyp} \label{Hyp:LambdaOmega} %Hypothesis: Lambda and Omega Functions
	The functions $\lambda$ and $\omega$ in $(\ref{LambdaOmegaLDS})$ satisfy the following:
	\begin{itemize}
	\item[{\rm(1)}] $\lambda : [0,\infty) \to \mathbb{R}$ is continuously differentiable and there exists some $a> 0$, with the property that $\lambda(a) = 0$ and $\lambda'(a) \neq 0$.
	\item[{\rm(2)}] $\omega = \omega(R,\alpha): [0,\infty) \times \mathbb{R} \to \mathbb{R}$ is continuously differentiable in both its arguments such that 
	\begin{equation}
		\omega(R,\alpha) - \omega(a,\alpha) = \alpha \omega_1(R,\alpha),
	\end{equation} 
	for some function $\omega_1(R,\alpha)$ which is continuously differentiable on the same domain with $\omega_1(a,\alpha) = 0$ for all $\alpha \in \mathbb{R}$. 
\end{itemize}
\end{hyp}

Conditions $(1)$ and $(2)$ of Hypothesis $\ref{Hyp:LambdaOmega}$ arise as natural generalizations of the normal form of a Hopf bifurcation and are similar to those which were assumed by Cohen et. al. and Greenberg in their respective proofs of spiral wave solutions in spatially-continuous reaction-diffusion equations $\cite{Cohen,Greenberg}$. More specific investigations often consider the Ginzburg-Landau equations with $\lambda(R) = \pm a^2 \mp R^2$, which clearly satisfies condition $(1)$ and $\omega(R,\alpha)$ to be a constant function independent of its arguments. When extending to non-constant functions $\omega(R,\alpha)$, in similar works exploring rotating waves in Lambda-Omega systems, these functions are taken to be slight perturbations of constant functions. Typically one considers 
\begin{equation} \label{OmegaFn}
	\omega(R) = \beta + \varepsilon R^2,
\end{equation}
where $\beta$ is a real-valued constant and $\varepsilon$ is a small parameter. Then writing $\varepsilon = \alpha \varepsilon'$ we recast $(\ref{OmegaFn})$ to satisfy the condition $(2)$ by observing that
\begin{equation}
	\begin{split}
		\omega(R) = \omega(R,\alpha) &= \beta + \alpha\varepsilon'R^2 \\
		&= \beta + \alpha\varepsilon'a^2 + \alpha[2\varepsilon' a(R - a) + \varepsilon'(R-a)^2],
	\end{split}
\end{equation} 
where we have $\omega(a,\alpha) = \beta + \alpha\varepsilon'a^2$ and $\omega_1(R,\alpha) = 2\varepsilon' a(R - a) + \varepsilon'(R-a)^2$, fitting the requirements of condition $(2)$. An important point to note is that in the anti-continuum limit $\alpha \to 0^+$ the function $\omega$ reduces to a constant function, which numerical investigations have found to be a necessary condition with regards to the system on a finite lattice $\cite{ErmentroutLambdaOmega}$. 

The most important characteristic of Hypothesis $\ref{Hyp:LambdaOmega}$ on the $\lambda$ and $\omega$ functions is that in the absence of coupling $(\alpha = 0)$ each component exhibits an identical periodic oscillation with amplitude $a$ and frequency $2\pi/\omega(a,0)$. Indeed, notice that when $\alpha = 0$ the elements of system $(\ref{LambdaOmegaLDS})$ completely decouple and are therefore acting independently of each other, resulting in the system
\begin{equation}
	\dot{z}_{i,j} = z_{i,j}[\lambda(|z_{i,j}|) + {\rm i}\omega(|z_{i,j}|,0)],
\end{equation}  
for each $(i,j) \in \mathbb{Z}^2$. Decomposing each $z_{i,j}$ into polar variables using the ansatz 
\begin{equation}
	z_{i,j}(t) = r_{i,j}(t)e^{{\rm i}\theta_{i,j}(t)}
\end{equation}
results in the set of ordinary differential equations
\begin{equation} \label{PolarUncoupled}
	\begin{split}
	\begin{aligned}
		&\dot{r}_{i,j} = r_{i,j}\lambda(r_{i,j}), \\
		&\dot{\theta}_{i,j} = \omega(r_{i,j},0),
	\end{aligned}
	\end{split}
\end{equation}
for each $(i,j)\in\mathbb{Z}^2$. Taking $r_{i,j} = a$ leads to a periodic solution of the form
\begin{equation} \label{UncoupledPeriodicSoln}
	z_{i,j}(t) = ae^{{\rm i}(\omega(a,0)t + \theta_{i,j}^0)},
\end{equation}
where $\theta_{i,j}^0 \in S^1$ is an initial phase value for each $(i,j) \in \mathbb{Z}^2$. Furthermore, the non-degeneracy condition $\lambda'(a) \neq 0$ guarantees that this solution is either locally attracting or repelling. Figure $\ref{fig:PeriodicSoln}$ gives characteristic phase portraits of the system $(\ref{PolarUncoupled})$ in the Cartesian plane upon writing $z_{i,j} = x + {\rm i}y$.

\begin{figure} %Figure: Periodic Solution in the Uncoupled Lattice
	\centering
		\includegraphics[height = 4cm]{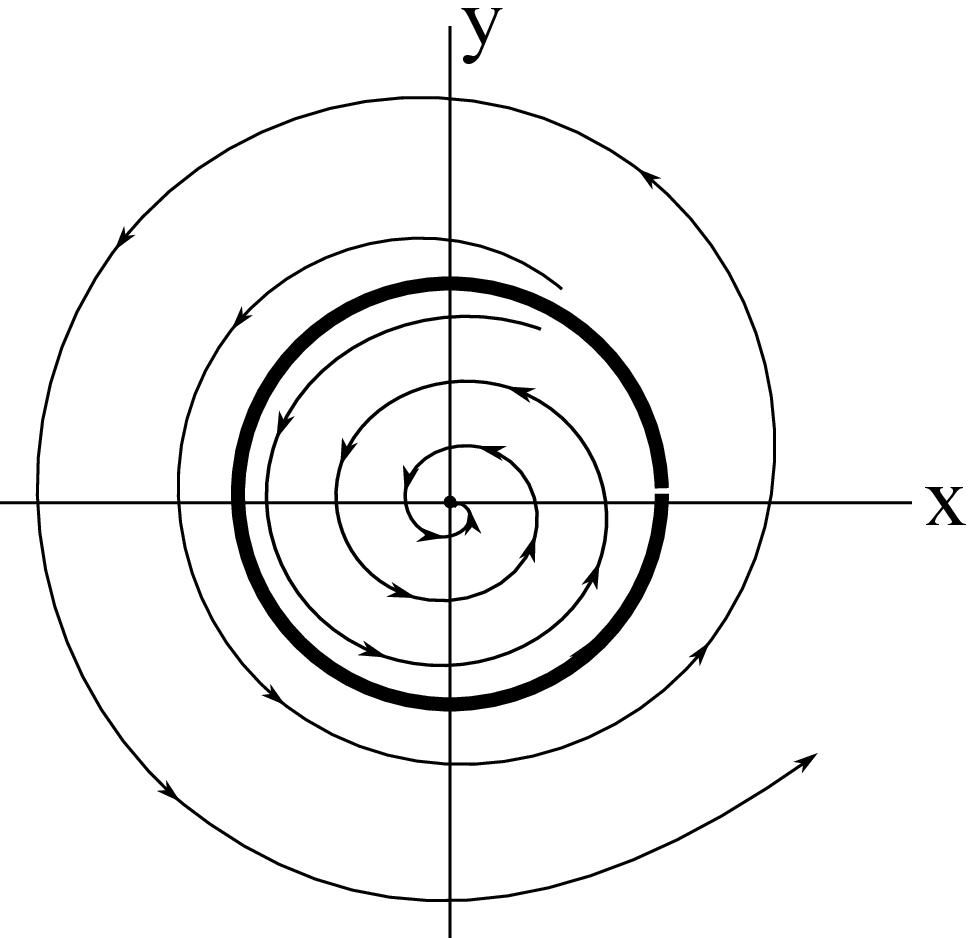}\ \ \ \ \ \ \ \ \ \ \ \ \ \ \ \ \ \
		\includegraphics[height = 4cm]{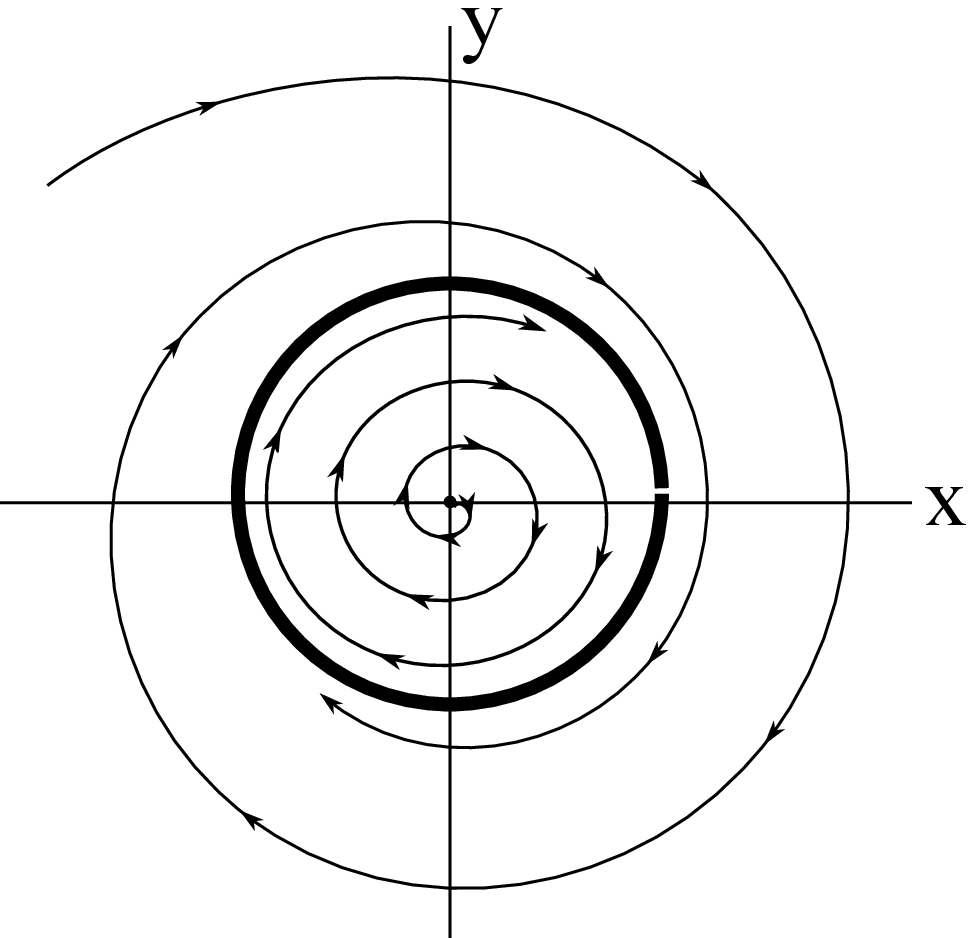}
		\caption{Typical phase portraits of the periodic solution $(\ref{UncoupledPeriodicSoln})$ and some nearby trajectories. There are two cases: (Left) locally repelling when $\lambda'(a) > 0$ and (Right) locally attracting when $\lambda'(a) < 0$.}
		\label{fig:PeriodicSoln}	
\end{figure} 

It is through Hypothesis $\ref{Hyp:LambdaOmega}$ and the previous discussion that the reader now can visualize our situation. In the absence of coupling there is the potential for all elements of the lattice dynamical system to be oscillating independently. Then, as $\alpha$ is slightly perturbed into positive values we seek to examine states of synchronized oscillation over the entire lattice. In particular our goal is to determine the existence of a state of synchronized oscillation which corresponds to a rotating wave in this discrete spatial context.

\subsection{Outline of the Paper} %Subsection: Outline of the Paper ---------------------------------------------------------------------------------------------------------------------

We begin with identifying all the functional analytic tools and nomenclature which will be required throughout this work in Section $\ref{sec:Preliminaries}$. It will be in Section $\ref{sec:Preliminaries}$ that we also present a non-standard Implicit Function Theorem, which comes as the main tool for obtaining the central result of this work. Following this delve into the abstract results used in this work we provide the appropriate definition of a rotating wave to our LDS $(\ref{LambdaOmegaLDS})$ in $(\ref{sec:Existence})$ along with a segmentation of this section into three subsections. The first subsection, Subsection $\ref{subsec:Abstract}$, moves the problem of identifying rotating waves in our lattice system to an abstract functional analytic setting where the tools of Section $\ref{sec:Preliminaries}$ can be applied. In Subsection $\ref{subsec:Abstract}$ we also comment on how much of the work carried out in the former investigation in Part I is of great use to our work here. Finally in Subsection $\ref{subsec:MainThm}$ we provide the main result of this work, whose proof is left to Section $\ref{sec:Proof}$. 

The proof of the main result of this paper is a careful and tedious application of a non-standard Implicit Function Theorem, which forms a large portion of this manuscript. The main proof of the existence of rotating wave solutions to $(\ref{LambdaOmegaLDS})$ is achieved by restricting our attention to the Lambda-Omega system in a `wedge' of indices belonging to the full lattice with appropriate boundary interactions, termed the reduced system. Upon restricting ourselves to this reduced system we define a mapping whose roots lie in one-to-one correspondence with solutions to the Lambda-Omega system restricted to our reduced system. To obtain roots to our mapping we provide several minor results and obtain a parameter continuation of a solution at $\alpha = 0$ by applying the non-standard Implicit Function Theorem introduced in Section $\ref{sec:Preliminaries}$. Upon provided a solution to the reduced Lambda-Omega system, one simply exploits the square symmetry of the lattice to obtain a solution over the full lattice with the desired characteristics to determine it to be a rotating wave in this context.

Upon providing the main existence result and its proof (Sections $\ref{sec:Existence}$ and $\ref{sec:Proof}$) we turn to a discussion of possible extensions and avenues for future work. This discussion is left to Section $\ref{sec:Discussion}$ where we contemplate the existence of multi-armed rotating waves, the stability of the solution provided here, and state some long-term goals aimed to examine bifurcations of rotating waves in the discrete spatial setting. It is the intention of this discussion to motivate subsequent studies into the differences in dynamics and bifurcations of rotating waves in the continuous versus the discrete spatial settings.

\section{Preliminaries} \label{sec:Preliminaries} %Section: Preliminaries ---------------------------------------------------------------------------------------------------------------------

The aim of this section is to provide the necessary nomenclature and results from infinite-dimensional Banach space theory in order to properly frame and convey the results of this work. The results in this section are not original to this work and therefore will be stated without proof. For a more complete introduction to these topics the reader is urged to consult $\cite{Rudin}$, and for a more in-depth inspection of linear operators acting between Banach spaces see $\cite{AdjointLemmas}$. 

To begin, let us consider a linear operator $T: X \to Y$ acting between Banach spaces. Then we will denote the operator norm of this linear operator to be 
\begin{equation} \label{OperatorNorm}
	\|T\|_{op} = \sup_{0 \neq x \in X} \frac{\|Tx\|_Y}{\|x\|_X},
\end{equation}
where $\|\cdot\|_X$ is the norm on $X$ and $\|\cdot\|_Y$ is the norm on $Y$. Although we will consider linear operators acting between many different Banach spaces, we will always use the notation $(\ref{OperatorNorm})$ to denote the norm of the operator.

The {\em dual} space of the Banach space $X$ is the set of all bounded linear operators $f: X \to \mathbb{R}$ equipped with the usual operator norm. The dual space is complete with respect to the operator norm, and is therefore a Banach space, here denoted $X^*$. Then, for a linear operator $T: X\to Y$, we define the {\em adjoint} of $T$, denoted $T^*$, to be the operator $T^*: Y^* \to X^*$ acting by
\begin{equation}
	T^*(f)(x) = f(Tx)	
\end{equation} 
for all $x \in X$ and $f \in Y^*$. In this way one can show that $\|T\|_{op} = \|T^*\|_{op}$ and if $T$ is invertible then $(T^{-1})^* = (T^*)^{-1}$. The following lemma connects some of our understandings between the actions of $T$ and $T^*$ and will become quite important in a later section of this work.

\begin{lem} [{\em $\cite{AdjointLemmas}$, \S II.3, Theorem II.3.7}] \label{lem:Adjoint}
	A linear operator $T$ has dense range if and only if $T^*$ is one-to-one. 
\end{lem} 

The Banach spaces which will be of primary interest throughout this work will be the sequence spaces indexed by a countably infinite index set $\mathcal{I}$. In particular, our attention will be focussed on two particular sequences spaces, denoted $\ell^1(\mathcal{I})$ and $\ell^\infty(\mathcal{I})$. They are defined as follows:
\begin{equation}
	\ell^1(\mathcal{I}) = \bigg\{x = \{x_n\}_{n \in \mathcal{I}}\ |\ \sum_{n \in \mathcal{I}} |x_n| < \infty\bigg\}	,
\end{equation}  
and 
\begin{equation}
	\ell^\infty(\mathcal{I}) = \bigg\{x = \{x_n\}_{n \in \mathcal{I}}\ |\ \sup_{n \in \mathcal{I}} |x_n| < \infty\bigg\}.
\end{equation}  
It is well-known that $\ell^1(\mathcal{I})$ is complete (and therefore a Banach space) under the norm
\begin{equation}
	\|x\|_1 = \sum_{n \in \mathcal{I}} |x_n|,
\end{equation}
and similarly the norm associated to $\ell^\infty(\mathcal{I})$ is given by
\begin{equation} \label{InfNorm}
	\|x\|_\infty = \sup_{n \in \mathcal{I}} |x_n|.
\end{equation}

Along with the sequence spaces $\ell^1(\mathcal{I})$ and $\ell^\infty(\mathcal{I})$, our attention will also focus on the closed subspace of $\ell^\infty(\mathcal{I})$ defined as 
\begin{equation} \label{c0Definition}
	c_0(\mathcal{I}) := \{x \in \ell^\infty| \forall \varepsilon > 0,\ \#\{n\in\mathcal{I}\ |\ |x_n| \geq \varepsilon\} < \infty\},
\end{equation}
where $\#\{\cdot\}$ has been used to denote the cardinality of the set. It is a straightforward exercise to see that $c_0(\mathcal{I})$ is a closed subspace of $\ell^\infty(\mathcal{I})$ with respect to the norm $(\ref{InfNorm})$, and is therefore a Banach space itself.  In the case when $\mathcal{I} = \mathbb{N}$ we have that $c_0(\mathbb{N})$ is merely the set of all sequences which converge to $0$. The definition provided in $(\ref{c0Definition})$ is a natural extension of this space to more diverse countable index sets.  An important characteristic of the space $c_0(\mathcal{I})$ is that its dual space, denoted $(c_0(\mathcal{I}))^*$, is isometrically isomorphic to $\ell^1(\mathcal{I})$, and can therefore be identified with this space. Hence, if $T: c_0(\mathcal{I}) \to c_0(\mathcal{I})$ is a bounded linear operator, then the corresponding dual operator $T^*$ is a bounded linear operator on $\ell^1(\mathcal{I})$.

The spaces $c_0(\mathcal{I})$ and $\ell^1(\mathcal{I})$ are both separable and exhibit a Shauder basis. In both spaces a Shauder basis is given by the canonical basis $\{\delta_n\}_{n \in \mathcal{I}}$, where $\delta_n$ is the sequence indexed by the elements of $\mathcal{I}$ with a $1$ at index $n$ and $0$'s everywhere else. Hence, for $X = c_0(\mathcal{I})$ or $\ell^1(\mathcal{I})$ and a linear operator $T:X \to X$ we can represent $T$ as an infinite matrix $T = [t_{nm}]_{n,m\in\mathcal{I}}$ by 
\begin{equation}
	t_{nm} = \langle T\delta_n,\delta_m \rangle,
\end{equation} 
where $\langle \cdot, \cdot \rangle$ represents the usual dot product given by the sum of component-wise multiplications. Hence, one can see that $t_{nm}$ is merely the element at the $m$th index of $T\delta_n$. Then using this notation one can show that $T^*:X^* \to X^*$ is given by the transpose of the infinite matrix $[t_{nm}]_{n,m\in\mathcal{I}}$.

We move now from linear operators to generally nonlinear operators acting between Banach spaces. In this work we will be interested in a slightly stronger version of Fr\'echet differentiability. That is, $F:X \to Y$ is said to be {\em strongly Fr\'echet differentiable} at the point $x_0$ if 
\begin{equation}
	\lim_{x_1,x_2 \to x_0} \frac{\|F(x_1) - F(x_2) - F'(x_0)(x_1 - x_2)\|_Y}{\|x_1 - x_2\|_X} = 0,	
\end{equation}   
where $F'(x_0)$ again denotes the Fr\'echet derivative of $F$ at $x = x_0$. The following theorem relates Fr\'echet differentiability to strong Fr\'echet differentiability. 

\begin{thm}[{\em $\cite{Schechter}$, \S 25, Theorem 25.23}] \label{thm:StrongDiff}
	If $F:X\to Y$ is differentiable on an open set, then $F$ is continuously differentiable if and only if $F$ is strongly differentiable on the same set. 
\end{thm} 

The central problem in this work stems from the fact that the Fr\'echet derivative of our nonlinear correspondence between Banach spaces fails to be invertible. To work around this failure of invertibility we introduce the notion of an approximate right inverse. Begin by again considering two arbitrary Banach spaces, $X$ and $Y$ with respective norms $\|\cdot\|_X$ and $\|\cdot\|_Y$, and let $D$ be a dense subset of $X$. A (not necessarily bounded) linear mapping $M:D \to Y$ is called {\em approximately right invertible} if, for each $\mu \in (0,1)$, there exists a norm $\|\cdot\|_\mu$ on $X$, a bounded mapping $B_\mu:Y \to X$, and a bound $\Gamma(\mu)$, depending on $\mu$, such that for all $y \in Y$ we have
\begin{equation}
	\|MB_\mu y - y\|_Y \leq \mu \|y\|_Y
\end{equation} 
and
\begin{equation}
	\|B_\mu y\|_\mu \leq \Gamma(\mu)\|y\|_Y,
\end{equation}
with the property that for all $x \in X$ we have
\begin{equation}
	\{\|x\|_\mu\} \nearrow \|x\|_X\ {\rm as}\ \mu \searrow 0.  
\end{equation}
Here we use the notation $\nearrow$ to denote monotonically increasing convergence and $\searrow$ as monotonically decreasing convergence. Then each $B_\mu$ is called an {\em approximate right inverse} of $M$. We denote $X_{\mu}$ to be the completion of $X$ with respect to the norm $\|\cdot\|_\mu$. Note that $B_\mu$ need not be linear, and particularly in the present situation it will not be.

Our work centres around applying the following theorem due to Craven and Nashed to our lattice dynamical system $(\ref{LambdaOmegaLDS})$ in order to determine the existence of rotating wave solutions for $\alpha > 0$.

\begin{thm}[{\em $\cite{CravenNashed}$, \S 3, Theorem 2}] \label{thm:CravenNashed}
	Let $X$ and $Y$ be real Banach spaces, with $a\in X$. Let $S$ be a closed convex cone in $Y$. Let the function $G:X \to Y$ be strongly Fr\'echet differentiable at $a$. Let $b:= G(a)$ and assume $b\in S$. Let the Fr\'echet derivative $M = G'(a):X \to Y$ be bounded linear with approximate right inverses $B_\mu$ and bound function $\Gamma(\mu) = k_0\mu^{-\gamma}$, with $\gamma < 1$ and $k_0 > 0$. Then for sufficiently small $\mu$, whenever $c$ satisfies $-[G(a) + G'(a)c]\in S$, and $\|c\|_{\mu} = 1$, there exists a solution $x = a +tc + \eta(t) \in X_{\mu}$ to $-G(x) \in S$, valid for all sufficiently small $t > 0$, with $x \neq a$. With an appropriate choice of $\mu = \mu(t) \to 0$ as $t \to 0^+$, $\|\eta(t)\|_{\mu(t)} = o(t)$ as $t \to 0^+$. 
\end{thm}

The original statement of Craven and Nashed's Implicit Function Theorem uses a weaker form of the derivative called the Hadamard derivative $\cite{CravenNashed}$. Since we will only be concerned with the stronger Fr\'echet derivative in this work, we merely restate the theorem with this form of differentiability. Since an appropriate reference was unable to be found, a proof that strong Fr\'echet differentiability at a point implies restricted strong Hadamard differentiability at the same point is provided in Appendix $\ref{sec:Appendix1}$. This therefore poses no problem with our restatement of the theorem here.

\section{Existence of Rotating Waves} \label{sec:Existence} %Section: Existence of Rotating Waves ---------------------------------------------------------------------------------------------------------------------

In Part I of this manuscript we introduced the notion of a rotating wave in this discrete spatial setting by considering the rotation operator acting on the indices of the lattice given by
\begin{equation} \label{RotationOperator}
	R(z_{i,j}) = z_{j,1-i}.
\end{equation}  
The effect this operator has on the closest cells to its centre of rotation is shown in Figure $\ref{fig:Rotation}$. One can see that we rotate the lattice clockwise through an angle of $\pi/2$ about a theoretical centre cell at $i=j=1/2$, which will act as the centre of rotation for our rotating wave solution. Although, it should be noted that due to the translational invariance of the lattice this centre can be chosen to lie between any square arrangement of cells and still give a rotating wave solution. We can therefore recall the following definition of a rotating wave in this setting, which was introduced in Part I. 

\begin{figure} %Figure: Rotation Operator
	\centering
	\includegraphics[height = 5cm]{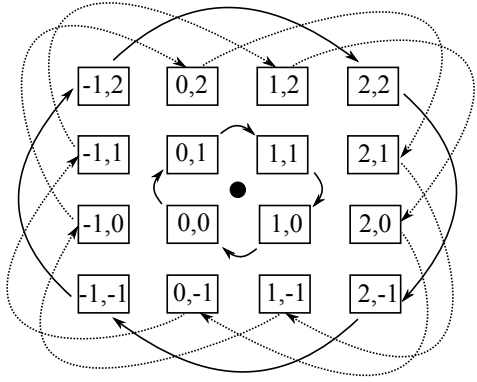}
	\caption{A diagram showing how the rotation operator defined in ($\ref{RotationOperator}$) effects elements of the lattice. The operator rotates lattice points by $\pi/2$ about a theoretical centre cell at $i=j=1/2$, represented by the dot in the centre of the diagram.}
	\label{fig:Rotation}
\end{figure}

\begin{defn} \label{def:RotWaveLDS}
	A {\bf rotating wave solution} of system $(\ref{LambdaOmegaLDS})$, denoted $\{z_{i,j}(t)\}_{(i,j)\in\mathbb{Z}^2}$, is a periodic solution with period $T > 0$ such that for all $(i,j) \in \mathbb{Z}^2$ and $t \in \mathbb{R}$ we have
	\begin{equation} \label{RotatingSymmetry}
		R(z_{i,j}(t)) = z_{i,j}(t + T/4).
	\end{equation}
\end{defn}

In this section we will present the main result of this work by demonstrating how solutions in a reduction of the lattice can be extended via symmetries to a solution over the full lattice which represent a rotating wave. We begin by moving our problem to an abstract setting in which rotating wave solutions to $(\ref{LambdaOmegaLDS})$ can be identified as roots to an infinite nonlinear system of equations. Upon moving to this abstract setting we will see exactly how the work of Part I can be applied to obtain roots of this infinite system of nonlinear systems at the single parameter value $\alpha = 0$. The combination of moving to the abstract setting and identifying a solution for a single parameter value then allows one to apply Theorem $\ref{thm:CravenNashed}$ and gain persistence of the solution into the parameter range $\alpha > 0$ and sufficiently small.

\subsection{Abstract Setting for the Problem} \label{subsec:Abstract}

Begin by introducing the ansatz 
\begin{equation} \label{PolarAnsatz}
	z_{i,j} = r_{i,j}e^{{\rm i}(\Omega t + \theta_{i,j})}, 
\end{equation}	
with $r_{i,j} =r_{i,j}(t)$ and $\theta_{i,j} = \theta_{i,j}(t)$ for each $(i,j)\in\mathbb{Z}^2$. Then the LDS ($\ref{LambdaOmegaLDS}$) can now be written in polar form as
\begin{equation} \label{FullPolarLattice}
	\begin{split}
		&\dot{r}_{i,j} = \alpha\sum_{i',j'} (r_{i',j'}\cos(\theta_{i',j'} - \theta_{i,j}) - r_{i,j}) + r_{i,j}\lambda(r_{i,j}), \\
		&\dot{\theta}_{i,j} =  \alpha\sum_{i',j'} \frac{r_{i',j'}}{r_{i,j}}\sin(\theta_{i',j'} - \theta_{i,j}) + (\omega(r_{i,j},\alpha) - \Omega), \ \ \ (i,j) \in \mathbb{Z}^2.
	\end{split}
\end{equation}    
Under Hypothesis $\ref{Hyp:LambdaOmega}$, without loss of generality we may assume $\Omega := \Omega(\alpha) = \omega(a,\alpha)$ for each $\alpha$. Indeed, if $\Omega \neq \omega(a,\alpha)$ then we may apply the linear change of variable 
\begin{equation}
	\hat{\theta}_{i,j}(t) = \theta_{i,j}(t) - (\omega(a,\alpha) - \Omega)t
\end{equation}
thereby reducing equations ($\ref{FullPolarLattice}$) to
\begin{equation} \label{FullPolarLattice2}
	\begin{split}
		&\dot{r}_{i,j} = \alpha\sum_{i',j'} (r_{i',j'}\cos(\hat{\theta}_{i',j'} - \hat{\theta}_{i,j}) - r_{i,j}) + r_{i,j}\lambda(r_{i,j}), \\
		&\dot{\hat{\theta}}_{i,j} =  \alpha\sum_{i',j'} \frac{r_{i',j'}}{r_{i,j}}\sin(\hat{\theta}_{i',j'} - \hat{\theta}_{i,j}) + \alpha\omega_1(r_{i,j},\alpha), \ \ \ (i,j) \in \mathbb{Z}^2. 
	\end{split}
\end{equation}  
Moreover, a solution (if it exists) becomes
\begin{equation}
	z_{i,j}(t) = r_{i,j}(t)e^{{\rm i}(\Omega t + \theta_{i,j}(t))} = r_{i,j}(t)e^{{\rm i}(\Omega t + \hat{\theta}_{i,j}(t) + (\omega(a,\alpha) - \Omega)t)} = r_{i,j}(t)e^{{\rm i}( \omega(a,\alpha) t + \hat{\theta}_{i,j}(t))}, 	
\end{equation} 
showing that we can take $\Omega = \omega(a,\alpha)$ without any loss of generality in the lattice system. 

It is system $(\ref{FullPolarLattice2})$ which will be of interest throughout this section and the next. Searching for nontritival steady-state solutions to the system $(\ref{FullPolarLattice2})$ with $\alpha \geq 0$ requires solving the nonlinear equations 
\begin{equation} \label{PolarRoots}
	\begin{split}
		&0 = \alpha\sum_{i',j'} (r_{i',j'}\cos(\theta_{i',j'} - \theta_{i,j}) - r_{i,j}) + r_{i,j}\lambda(r_{i,j}), \\
		&0 =  \sum_{i',j'} \frac{r_{i',j'}}{r_{i,j}}\sin(\theta_{i',j'} - \theta_{i,j}) + \omega_1(r_{i,j},\alpha),  
	\end{split}	
\end{equation} 
upon dropping the hats, for all $(i,j) \in \mathbb{Z}^2$. Then one sees that solving these nonlinear equations for nontrivial steady-states results in a periodic solution $\{z_{i,j}(t)\}_{(i,j)\in\mathbb{Z}^2}$ of the form $(\ref{PolarAnsatz})$ where each element of the lattice is oscillating with a frequency of $2\pi/\omega(a,\alpha)$. In order to solve these equations we present the following definition.

\begin{defn} \label{defn:ReducedSystem} %Definition of the Reduced System
	We refer to the {\bf reduced system} as the lattice dynamical system $(\ref{FullPolarLattice2})$ restricted to the indices $\Lambda \subset \mathbb{Z}^2$ given by
	\begin{equation} \label{Lambda}
		\Lambda = \{(i,j) \in \mathbb{Z}^2|\ i\geq 1\ {\rm and} \ 2 - i \leq j \leq i\}.
	\end{equation} 
	along with the boundary conditions
	\begin{equation} \label{ReducedBCs}
		\begin{split}
		&\bullet R(r_{i,i}) = r_{i,i}, \\
		&\bullet R(\theta_{i,i}) = \theta_{i,i} + \frac{\pi}{2}, \\
		&\bullet R(r_{i+1,i-1)} = r_{i+1,1-i},\\
		&\bullet R(\theta_{i+1,1-i}) = \theta_{i+1,1-i} + \frac{3\pi}{2}.
		\end{split}
	\end{equation}
\end{defn}

\begin{rmk}
	 The indices of the reduced system, denoted by $\Lambda$, are represented in Figure $\ref{fig:Lattice}$ by the shaded cells for visual reference. The key point here is that 
	\begin{equation}
		\mathbb{Z}^2 = \Lambda \sqcup R(\Lambda) \sqcup R^2(\Lambda) \sqcup R^3(\Lambda),
	\end{equation}
	and hence we have partitioned the full lattice into four mutually disjoint subsets. A steady-state solution in the reduced system will then be used to determine a solution over the entire lattice by forcing the 	rotational symmetry condition $(\ref{RotatingSymmetry})$, thus obtaining a solution to the full system with the required symmetry. The boundary conditions $(\ref{ReducedBCs})$ give the rotational symmetry requirement $R(\theta_{i,i}) = \theta_{i,i} + \frac{\pi}{2}$ along with the requirement that $R^3(\theta_{1+i,1-i}) = \theta_{1+i,1-i} + \frac{3\pi}{2}$. Moreover, when viewing a single column in $\Lambda$ we find that 	these boundary conditions impose that the top and bottom of each column are linked. For example, the differential equation at the index $(1,1)$ becomes before canceling terms 
\begin{equation}
\begin{split}
  \dot{\theta}_{1,1} &
= \underbrace{\sin(\theta_{2,1} - \theta_{1,1})}_\text{Right}
+ \underbrace{\sin(\theta_{1,1} + \frac{\pi}{2} - \theta_{1,1})}_\text{Down} 
+ \underbrace{\sin(\theta_{1,1} + \frac{3\pi}{2} -
  \theta_{1,1})}_\text{Left}  \\
&+ \underbrace{\sin(\theta_{2,0} + \frac{3\pi}{2} - \theta_{1,1})}_\text{Up}.
\end{split}
\end{equation}
	The boundary conditions on the radial components follow in a similar manner in that $R(r_{i,i}) = r_{i,i}$ and $R^3(r_{1+i,1-i}) = r_{1+i,1-i}$. Figure $\ref{fig:ReducedGraph}$ provides a visual representation of the connections in $\Lambda$ prescribed by Definition $\ref{defn:ReducedSystem}$.  
\end{rmk}

\begin{figure} %Figure: Reduced Graph
	\centering
	\includegraphics[height = 6cm]{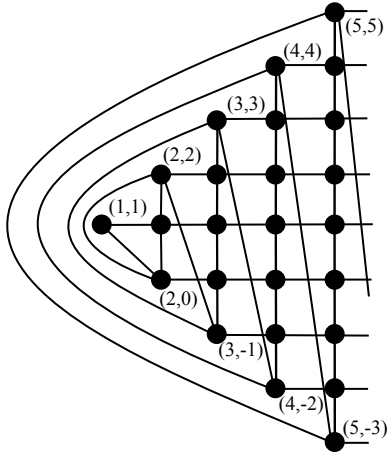}
	\caption{The connections of the reduced system. Here the dots represent the elements at each index $(i,j)\in\Lambda$ and the lines connecting these dots represent the connections of the reduced system per Definition $\ref{defn:ReducedSystem}$.}
	\label{fig:ReducedGraph}
\end{figure} 

Abstractly, solving for steady-states to the reduced system requires obtaining zeros of the infinite set of nonlinear equations given by  
\begin{equation} \label{FMapping}
	\begin{split}
		F_{i,j}^1(\alpha, r,\theta) = \alpha\sum_{i',j'}& [r_{i',j'}\cos(\theta_{i',j'} - \theta_{i,j}) - r_{i,j}] + r_{i,j}\lambda(r_{i,j}), \\
		F_{i,j}^2(\alpha, r,\theta) = \sum_{i',j'}& \frac{r_{i',j'}}{r_{i,j}}\sin(\theta_{i',j'} - \theta_{i,j}) + \omega_1(r_{i,j},\alpha),
	\end{split}
\end{equation} 
for all $(i,j) \in \Lambda$ along with the boundary conditions $(\ref{ReducedBCs})$. When $\alpha = 0$ we obtain 
\begin{equation}
	F_{i,j}^1(0,r,\theta) = r_{i,j}\lambda(r_{i,j}),	
\end{equation}  
which from condition $(1)$ of Hypothesis $\ref{Hyp:LambdaOmega}$ we have that a nontrivial solution exists given by $r_{i,j} = a > 0$ for all $(i,j)\in\Lambda$. Clearly this solution satisfies the boundary conditions of the reduced system since all elements are identical. We denote the element ${\bf a} = \{a\}_{(i,j) \in \Lambda}$. Evaluating the second components of $(\ref{FMapping})$ at $\alpha = 0$ and $r = {\bf a}$ results in
\begin{equation} \label{InfinitePhase}
	F_{i,j}^2(0,{\bf a},\theta) = \sum_{i,j}\sin(\theta_{i',j'} - \theta_{i,j}), 
\end{equation}  
since condition $(2)$ of Hypothesis $\ref{Hyp:LambdaOmega}$ implies that $\omega_1(a,0) = 0$. 

In $\cite{MyWork}$ it was shown that the system of equations
\begin{equation} \label{InfinitePhase2}
	0 = \sum_{i,j}\sin(\theta_{i',j'} - \theta_{i,j})
\end{equation}
over all indices $(i,j) \in \mathbb{Z}^2$ possesses a nontrivial solution which corresponds to the phase lags of a rotating wave. That is, if we denote this solution $\{\bar{\theta}_{i,j}\}_{(i,j)\in\mathbb{Z}^2}$, then it was shown that 
\begin{equation}
	R(\bar{\theta}_{i,j}) = \bar{\theta}_{i,j} + \frac{\pi}{2},
\end{equation}   
for all $(i,j) \in \mathbb{Z}^2$. Therefore, considering $\bar{\theta} = \{\bar{\theta}_{i,j}\}_{(i,j)\in\Lambda}$ we in turn have that $F_{i,j}^2(0,{\bf a},\bar{\theta}) = 0$ for all $(i,j) \in \Lambda$. Coupling this with the results above implies that we have therefore identified a nontrivial root of the infinite system of nonlinear equations $\{(F^1_{i,j},F^2_{i,j})\}_{(i,j) \in \Lambda}$, and our goal is now to extend this solution to a solution for small $\alpha > 0$.

\subsection{Statement of the Main Result} \label{subsec:MainThm}

The discussion in the previous subsection immediately leads to the following result.

\begin{thm} \label{thm:ReducedSolution} %Main Theorem
	Let $\lambda$ and $\omega$ be functions satisfying Hypothesis $\ref{Hyp:LambdaOmega}$. For $\alpha > 0$ and sufficiently small, there exists uniformly bounded $r(\alpha) = \{r_{i,j}(\alpha)\}_{(i,j)\in\Lambda}$ and $\theta(\alpha) = \{\theta_{i,j}(\alpha)\}_{(i,j)\in\Lambda}$ steady-state solutions to the reduced system. That is, these functions satisfy 
	\begin{equation}
		\begin{split}
			&F_{i,j}^1(\alpha,r(\alpha),\theta(\alpha)) = 0 \\
			&F_{i,j}^2(\alpha,r(\alpha),\theta(\alpha)) = 0
		\end{split}
	\end{equation}
	for all $(i,j)\in \Lambda$ and the boundary conditions $(\ref{ReducedBCs})$. As $\alpha \to 0^+$ we have $r(\alpha) \to {\bf a}$ uniformly and $\theta(\alpha) \to \bar{\theta}$. 
\end{thm}

The proof of Theorem $\ref{thm:ReducedSolution}$ is a meticulous application of the Theorem $\ref{thm:CravenNashed}$ and will be undertaken in the following section. Prior to proving Theorem $\ref{thm:ReducedSolution}$ we provide the main result of the paper, which comes as a direct consequence of the previous result. 

\begin{cor} \label{thm:RotatingExistence} %Existence of Rotating Waves Theorem
	Let $\lambda$ and $\omega$ be functions satisfying Hypothesis $\ref{Hyp:LambdaOmega}$. Then for $\alpha > 0$ sufficiently small there exists a nontrivial uniformly bounded rotating wave solution to the lattice dynamical system $(\ref{LambdaOmegaLDS})$ rotating with frequency $2\pi/\omega(a,\alpha)$. 
\end{cor}

\begin{proof}
	The solution of the reduced system from Theorem $\ref{thm:ReducedSolution}$ extends to a solution of the entire lattice via a set of transformations which are derived to correspond to Definition $\ref{def:RotWaveLDS}$ of a rotating wave in this framework. Using the rotation operator defined in $(\ref{RotationOperator})$ we apply the following transformations over the other three distinct regions of the lattice:
\begin{equation} \label{SymmetryExtensions2}	
	\begin{array}{ccc}
		{\rm\bf \underline{Region}} & {\rm\bf \underline{Radial\ Transformation}} & {\rm\bf \underline{Phase\ Transformation}} \\
		\\
		R(\Lambda) & R(r_{i,j}(\alpha)) = r_{i,j}(\alpha) & R(\theta_{i,j}(\alpha)) = \theta_{i,j}(\alpha) + \frac{\pi}{2}\\
		\\
		R^2(\Lambda) & R^2(r_{i,j}(\alpha)) = r_{i,j}(\alpha) & R^2(\theta_{i,j}(\alpha)) = \theta_{i,j}(\alpha) + \pi\\
		\\
		R^3(\Lambda) & R^3(r_{i,j}(\alpha)) = r_{i,j}(\alpha) & R^3(\theta_{i,j}(\alpha)) = \theta_{i,j}(\alpha) + \frac{3\pi}{2} 
	\end{array} 
\end{equation}
It is straightforward to see that these transformations give a solution over the entire lattice since for any elements in the interior of one of the three regions we have that the radial components remain the same as those in the reduced system and the phase components are all translated by exactly the same value, thus leaving their difference unchanged. The interactions with the boundary are taken care of by the boundary conditions imposed on the reduced system. Figure $\ref{fig:RotatingSoln}$ shows the four distinct regions of the lattice and how the solution of the reduced system can be extended to a solution over the entire lattice. Writing 
\begin{equation}
	z_{i,j}(t,\alpha) = r_{i,j}(\alpha)e^{{\rm i} (\omega(a,\alpha)t + \theta_{i,j}(\alpha))}
\end{equation}
for each $(i,j)$ allows one to easily check that
\begin{equation}
	\begin{split}
		R(z_{i,j}(t,\alpha)) &= r_{i,j}(\alpha)e^{{\rm i} (\omega(a,\alpha)t + \theta_{i,j}(\alpha) + \frac{\pi}{2})} \\
		&=r_{i,j}(\alpha)e^{{\rm i} (\omega(a,\alpha)(t + \frac{T_\alpha}{4}) + \theta_{i,j}(\alpha))} \\
		&=z_{i,j}(t + T_\alpha/4,\alpha),
	\end{split}
\end{equation} 
where $T_\alpha = 2\pi/\omega(a,\alpha)$ is the period of the solution for each $\alpha$. Hence, the symmetry requirement $(\ref{RotatingSymmetry})$ has been met by definition of the extensions, thus giving a rotating wave solution to the lattice dynamical system $(\ref{LambdaOmegaLDS})$. Uniform boundedness follows from Theorem $\ref{thm:ReducedSolution}$ and by the extensions over the entire lattice, thus completing the proof of the theorem. 

\begin{figure} %Figure: Extending from the Reduced System
	\centering
	\includegraphics[height=7cm, width = 7cm]{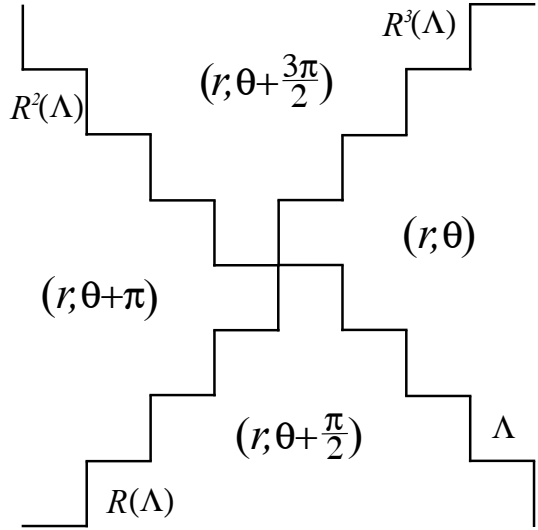}
	\caption{Extension of the solution in the reduced system over the entire lattice using the rotation operator.}
	\label{fig:RotatingSoln}
\end{figure}  

\end{proof}

\begin{rmk}
In the case that $\omega(\cdot,\alpha)$ is a constant function (i.e. $\omega_1 \equiv 0$ in $(\ref{PolarRoots})$) the reduced system $\Lambda$ can be further partitioned to focus our attention simply on those indices satisfying $1 \leq j \leq i$. This method was employed in Part I of this investigation where the odd symmetry of the sine function could be exploited. In this case the phase solutions of system $(\ref{LambdaOmegaLDS})$ would exhibit the same symmetry as that given for the solution $\bar{\theta}$, as discussed Part I. In this way we find that the effect a non-constant $\omega(\cdot,\alpha)$ has on the system is to break the odd symmetry in the differential equations of the phase components. It was shown by Ermentrout and Paullet in $\cite{ErmentroutLambdaOmega}$ that these non-constant $\omega(\cdot,\alpha)$'s can induce the familiar spiral wave spatial pattern, thus breaking the symmetry of the solution $\bar{\theta}$. 
\end{rmk}

\section{Proof of Theorem \ref{thm:ReducedSolution}} \label{sec:Proof} %Section: Proof of Existence Result ---------------------------------------------------------------------------------------------------------------------

Our goal in this section is to prove Theorem $\ref{thm:ReducedSolution}$ through an application of Theorem $\ref{thm:CravenNashed}$. We seek to define a proper mapping for use in Theorem $\ref{thm:CravenNashed}$ whose roots lie in correspondence with roots of the mapping $(\ref{FMapping})$. Our interest in this section is the product of Banach spaces
\begin{equation}
	X := \mathbb{R} \times \ell^\infty(\Lambda) \times c_0(\Lambda),	
\end{equation}
where we recall that $c_0(\Lambda)$ is the closed subspace of $\ell^\infty(\Lambda)$ defined in $(\ref{c0Definition})$. We further equip $X$ with the norm
\begin{equation}
	\|(\alpha,s,\psi)\|_{X} := \max\{|\alpha|, \|s\|_\infty, \|\psi\|_\infty\}.
\end{equation}
The fact that $X$ is complete with respect to this norm follows from the completion of each of the components $\mathbb{R}$, $\ell^\infty(\Lambda)$ and $c_0(\Lambda)$ with respect to their norms, thus giving that $X$ is a Banach space.

Recalling from Hypothesis $\ref{Hyp:LambdaOmega}$ that $a > 0$ is a root of $\lambda$, we denote $B_{\frac{a}{2}}(0)$ to be the ball of radius $a/2$ centred at $0$ in $\ell^\infty(\Lambda)$. Then consider the mapping
\begin{equation} \label{GMappingSpaces}
	G= (\alpha, G^1, G^2)^T:  \mathbb{R} \times B_{\frac{a}{2}}(0) \times c_0(\Lambda) \subset X \to X
\end{equation} 
with components $G^1$ and $G^2$ given by
\begin{equation} \label{GMapping}
	\begin{split}
		G_{i,j}^1(\alpha,s,\psi) &=\alpha\sum_{i',j'} [(a + s_{i',j'})\cos(\bar{\theta}_{i',j'} + \psi_{i',j'} - \bar{\theta}_{i,j} - \psi_{i,j}) - (a+ s_{i,j})\bigg]\\
		 &\ \ \ \ \ \ \ \ \ \ \ \ \ \ \ \ \ \ \ \ \ \  + (a + s_{i,j})\lambda(a + s_{i,j}), \\
		G_{i,j}^2(\alpha,s,\psi) &= i^{-1}\bigg[\sum_{i',j'} \frac{(a + s_{i',j'})}{(a + s_{i,j})}\sin(\bar{\theta}_{i',j'} + \psi_{i',j'} - \bar{\theta}_{i,j} - \psi_{i,j})\\ 
		&\ \ \ \ \ \ \ \ \ \ \  \ \ \ \ \ \ \ \ \ \ \ + \omega_1(a + s_{i,j},\alpha)\bigg].
	\end{split}
\end{equation}
The elements $s$ and $\psi$ can be seen as the deviation from the solution when $\alpha = 0$, and from here it is easy to see that since $i \geq 1$ the roots of $G^1$ and $G^2$ lie in one-to-one correspondence with those of $(\ref{FMapping})$. Indeed, notice that for all $(i,j)\in\Lambda$ we have
\begin{equation}
	\begin{split}
		&G^1_{i,j}(\alpha,s,\psi) = F^1_{i,j}(\alpha,{\bf a} + s, \bar{\theta}+ \psi), \\
		&G^2_{i,j}(\alpha,s,\psi) = i^{-1}F^2_{i,j}(\alpha,{\bf a} + s, \bar{\theta}+ \psi)	
	\end{split}
\end{equation} 
In this way one should notice that the only difference between $(\ref{FMapping})$ and $G^1$, $G^2$ is that now $G^2$ has a decay term added to it, which we will prove acts to guarantee that this mapping is well-defined. 

We define the closed convex cone 
\begin{equation} \label{ConvexCone}
	S:= \mathbb{R} \times \{0\} \times \{0\} \subset X. 
\end{equation}
One sees from our work in the previous section that $G^1(0,0,0) = 0$ and $G^2(0,0,0) = 0$, giving that $G(0,0,0) = (0,0,0) \in S$. Therefore, proving Theorem $\ref{thm:ReducedSolution}$ is reduced to finding elements $(\alpha,s,\psi)$ with $\alpha \neq 0$ such that $G(\alpha,s,\psi) \in S$ since this will imply that $G^1(\alpha,s,\psi) = 0$ and $G^2(\alpha,s,\psi) = 0$. 

\begin{prop}
	The mapping $G$ as defined in $(\ref{GMapping})$ is a well-defined operator from its domain into $X$.  		
\end{prop}

\begin{proof}
	Clearly the first component of $G$ is well-defined, and therefore we need only check that the components $G^1$ and $G^2$ map into $\ell^\infty(\Lambda)$ and $c_0(\Lambda)$, respectively.

	We begin by showing that $G^1$ maps into $\ell^\infty(\Lambda)$. Let $(\alpha,s,\psi) \in X$ be fixed with $\|s\|_\infty < \frac{a}{2}$. Let
	\begin{equation}
		C_\lambda := \max_{R \in [-a/2,a/2]} |\lambda(a + R)|.
	\end{equation}
	Then for each $(i,j)\in\Lambda$ we have
	\begin{equation}
		\begin{split}
			|G^1_{i,j}(\alpha,s,\psi)| &\leq |\alpha|\sum_{i',j'} (|a + s_{i',j'}| + |a+ s_{i,j}|) + |a + s_{i,j}\|\lambda(a + s_{i,j})|  \\
			&\leq 8|\alpha| a + 8|\alpha| \|s\|_\infty + (a + \|s\|_\infty)C_\lambda  \\
			&\leq 12|\alpha| a + \frac{3}{2}aC_\lambda,
		\end{split}
	\end{equation}
	where we have used the fact that each element has at most four nearest-neighbours. Taking the supremum over all $(i,j)\in\Lambda$ gives that
	\begin{equation}
		\|G^1(\alpha,s,\psi)\|_\infty = \sup_{(i,j)\in\Lambda} |G^1_{i,j}(\alpha,s,\psi)| \leq 12|\alpha| a + \frac{3}{2}aC_\lambda < \infty,  
	\end{equation}
	thus showing that $G^1$ is a well-defined mapping.
	
	We now turn to $G^2$. Again let $(\alpha,s,\psi) \in X$ with $\|s\|_\infty < a/2$. Notice that 
	\begin{equation}
		\bigg|\frac{(a + s_{i',j'})}{(a + s_{i,j})}\sin(\bar{\theta}_{i',j'} + \psi_{i',j'} - \bar{\theta}_{i,j} - \psi_{i,j})\bigg| \leq 3	
	\end{equation}
	for all $(i,j)\in\Lambda$ since $|s_{i,j}|,|s_{i',j'}| < \frac{a}{2}$. For the specific $\alpha$ in question, let us denote 
	\begin{equation}
		C_\omega(\alpha) := \max_{R \in [-a/2,a/2]} |\omega_1(a + R,\alpha)|.
	\end{equation} 
	Then, putting this all together gives that
	\begin{equation}
		\begin{split}
			|G^2_{i,j}(\alpha,s,\psi)| &\leq i^{-1}\bigg[\sum_{i',j'}\bigg|\frac{(a + s_{i',j'})}{(a + s_{i,j})}\sin(\bar{\theta}_{i',j'} + \psi_{i',j'} - \bar{\theta}_{i,j} - \psi_{i,j})\bigg| + |\omega_1(a+s_{i,j},\alpha)|\bigg]  \\
			&\leq i^{-1}[12 + C_\omega(\alpha)].
		\end{split}
	\end{equation}
	Now let $\varepsilon > 0$ be arbitrary and let $i_0 \geq 1$ be the smallest integer such that
	\begin{equation}
		i_0 > \varepsilon^{-1}[12 + C_\omega(\alpha)]. 
	\end{equation} 
	Therefore, for all $i \geq i_0$ we have that 
	\begin{equation}
		|G^2_{i,j}(\alpha,s,\psi)| \leq i^{-1}[12 + C_\omega(\alpha)] < \varepsilon,	
	\end{equation}
	showing that the only indices $(i,j)$ such that $G_{i,j}^2(\alpha,s,\psi)$ which can exceed $\varepsilon$ are those with $i < i_0$, which is a finite subset of $\Lambda$. Hence, $G^2$ is well-defined. 
\end{proof} %End of Well-Defined Proof

We now turn to analyzing the strong Fr\'echet differentiability of the mapping $G$ at $(\alpha,s,\psi) = (0,0,0)$. 

\begin{prop} \label{prop:FrechetDerivative}
	$G(\alpha,s,\psi)$ is strongly Fr\'echet differentiable at $(\alpha,s,\psi) = (0,0,0)$. That is, the Fr\'echet derivative at this point exists, is a bounded linear operator and can be written in the block matrix form as 
	\begin{equation} \label{DerivativeMatrix}
	G'(0,0,0) := M = \begin{bmatrix}
		1 & {\bf 0} & {\bf 0}\\
		M_{21} & M_{22} & {\bf 0} \\
		{\bf 0} & M_{32} & M_{33}
	\end{bmatrix},
	\end{equation}
	where ${\bf 0}$ represents the trivial operator which sends every element to the $0$ of the appropriate space. The operators have the following specific forms:
	\begin{enumerate}
		\item $M_{21} :\mathbb{R} \to \ell^\infty(\Lambda)$ is the bounded linear operator acting by
			\begin{equation} \label{M21}
				[M_{21} \alpha]_{i,j} = a\alpha\sum_{i',j'} [\cos(\bar{\theta}_{i',j'} - \bar{\theta}_{i,j}) - 1],
			\end{equation} 
		for all $(i,j)\in\Lambda$.
		\item $M_{22}:\ell^\infty(\Lambda) \to \ell^\infty(\Lambda)$ is the bounded linear operator acting by
			\begin{equation} \label{M22}
				[M_{22}s]_{i,j} = a\lambda'(a)s_{i,j}
			\end{equation} 
		\item $M_{32}:\ell^\infty(\Lambda) \to c_0(\Lambda)$ is the bounded linear operator acting by
			\begin{equation} \label{M32}
				[M_{32} s]_{i,j} = i^{-1}\bigg[\frac{1}{a}\sum_{i',j'} s_{i',j'} \sin(\bar{\theta}_{i',j'} - \bar{\theta}_{i,j}) + s_{i,j} \partial_1 \omega_1(a + R,\alpha)\bigg|_{(R,\alpha) = (0,0)}\bigg],
			\end{equation} 
			for all $(i,j)\in\Lambda$, where $\partial_1$ denotes the partial derivative with respect to the first argument. 
		\item $M_{33}:c_0(\Lambda) \to c_0(\Lambda)$ is the bounded linear operator acting by
			\begin{equation} \label{M33}
				[M_{33}\psi]_{i,j} = i^{-1}\sum_{i',j'} \cos(\bar{\theta}_{i',j'} - \bar{\theta}_{i,j})(\psi_{i',j'} - \psi_{i,j})	
			\end{equation}
			for all $(i,j)\in \Lambda$.
	\end{enumerate} 
\end{prop}

The proof of Proposition $\ref{prop:FrechetDerivative}$ is left to Appendix $\ref{sec:Appendix2}$ to maintain that the results of this section are easy to follow. 

\begin{lem} \label{lem:SPartial} %Partial derivative of G^1  
	The operator $M_{22}:\ell^\infty(\Lambda) \to \ell^\infty(\Lambda)$ is injective, surjective and has a bounded inverse. 
\end{lem} 

\begin{proof}
{\em \underline{Injectivity:}} From Hypothesis $\ref{Hyp:LambdaOmega}$ we have that $\lambda'(a) \neq 0$, and therefore $M_{22}$ is a nontrivial diagonal operator. Then one sees that $M_{22} s = 0$ gives for each $(i,j) \in \Lambda$ that
	\begin{equation} \label{M22Injective}
		[M_{22} s]_{i,j} = a\lambda'(a) s_{i,j} = 0 \implies s_{i,j} = 0,
	\end{equation}
	thus giving that $s = 0$ and hence that $M_{22}$ is injective. 
		
	{\em \underline{Surjectivity:}} Consider the element $y = \{y_{i,j}\}_{(i,j)\in\Lambda} \in \ell^\infty(\Lambda)$. Then clearly the element $s = \{[a\lambda'(a)]^{-1}y_{i,j}\}_{(i,j)\in\Lambda}$ belongs to $\ell^\infty(\Lambda)$ with norm given by $\|s\|_\infty = [a\lambda'(a)]^{-1}\|y\|_\infty < \infty$ and satisfies $M_{22}s = y$.  
	
{\em \underline{Bounded Inverse:}} This is an immediate
consequence of the fact that\\
$M_{22}:\ell^\infty(\Lambda) \to \ell^\infty(\Lambda)$ is a bounded bijection.  
\end{proof} %End of Proof of M_{22}

To obtain comparable results for $M_{33}$ we first provide the following lemma. 

\begin{lem} \label{lem:OperatorT} %Operator T Lemma
	Consider the operator $T: c_0(\Lambda) \to c_0(\Lambda)$ acting by
	\begin{equation} \label{OperatorT}
		[T_{\bar{\theta}}\psi]_{i,j} = \sum_{i',j'} \cos(\bar{\theta}_{i',j'} - \bar{\theta}_{i,j})(\psi_{i',j'} - \psi_{i,j})	
	\end{equation}
	for all $(i,j)\in \Lambda$. Then $T_{\bar{\theta}}$ is injective and has dense range.  
\end{lem}

\begin{proof}
	{\em \underline{Injectivity:}} It was demonstrated in $\cite{MyWork}$ that all nearest-neighbour interactions in $\Lambda$ of the element $\bar{\theta}$ are such that $|\bar{\theta}_{i',j'} - \bar{\theta}_{i,j}| < \frac{\pi}{2}$ with the exceptions of $|\bar{\theta}_{1,0} - \bar{\theta}_{1,1}|$ and $|\bar{\theta}_{0,1} - \bar{\theta}_{1,1}|$ both taking the value $\frac{\pi}{2}$, and so $\cos(\bar{\theta}_{i',j'} - \bar{\theta}_{i,j}) > 0$ for all but two neighbouring interactions in $\Lambda$. 
	
	Now let us assume that $\psi = \{\psi_{i,j}\}_{(i,j)\in\Lambda}$ is such that $T_{\bar{\theta}} \psi =  0$. Then $[T_{\bar{\theta}}\psi]_{i,j} = 0$ for each $(i,j) \in \Lambda$ and we have
	\begin{equation} \label{QuadraticForm}
		0 = \sum_{(i,j)\in\Lambda} 0\cdot \psi_{i,j} =\sum_{(i,j)\in\Lambda} [T_{\bar{\theta}}\psi]_{i,j}\cdot \psi_{i,j} =\frac{1}{2} \sum_{(i,j)\in\Lambda} \sum_{i',j'} -\cos(\bar{\theta}_{i',j'} - \bar{\theta}_{i,j})(\psi_{i',j'} - \psi_{i,j})^2.  
	\end{equation}  
	To justify this notice that $\cos(\bar{\theta}_{i',j'} - \bar{\theta}_{i,j})$ contributes to the sum twice by the even symmetry of the cosine function. This gives for any index $(i,j)\in\Lambda$ 
	\begin{equation}
		\begin{split}
			\cos(\bar{\theta}_{i',j'} - \bar{\theta}_{i,j})(\psi_{i',j'} - \psi_{i,j})\psi_{i,j} + &\cos(\bar{\theta}_{i,j} - \bar{\theta}_{i',j'})(\psi_{i,j} - \psi_{i',j'})\psi_{i',j'} \\
			&= \cos(\bar{\theta}_{i',j'} - \bar{\theta}_{i,j})(\psi_{i',j'} - \psi_{i,j})(\psi_{i,j} - \psi_{i',j'}) \\
			&= -\cos(\bar{\theta}_{i',j'} - \bar{\theta}_{i,j})(\psi_{i',j'} - \psi_{i,j})^2.
		\end{split}	
	\end{equation}
	Furthermore, the factor of one half in front of the sum comes from the fact that everything is being summed twice, again by the symmetry of the cosine function. Therefore $(\ref{QuadraticForm})$ returns nonpositive values for any sequence indexed by the elements of $\Lambda$ since $\cos(\bar{\theta}_{i',j'} - \bar{\theta}_{i,j}) \geq 0$ for all $(i,j)\in\Lambda$ and its nearest-neighbours, $(i',j')$.
	
	Thus, in order for $(\ref{QuadraticForm})$ to hold, it must be the case that each $\psi_{i,j}$ is equal to its nearest-neighbours in $\Lambda$. This is because the only case when $\cos(\bar{\theta}_{i',j'} - \bar{\theta}_{i,j}) = 0$ is at the index $(1,1)$ and these represent connections with the boundary. Since the element at $(1,1)$ is connected to the element indexed by $(1,2)$, which is in $\Lambda$, we have that every index is connected to at least one other index in $\Lambda$ by some non-zero weight $\cos(\bar{\theta}_{i',j'} - \bar{\theta}_{i,j}) > 0$. This allows one to see that $T_{\bar{\theta}}\psi = 0$ if and only if there exists a $C \in \mathbb{R}$ such that $\psi_{i,j} = C$ for all $(i,j) \in \Lambda$. By definition these constant sequences are not elements of $c_0(\Lambda)$, thus giving that $T_{\bar{\theta}} : c_0(\Lambda) \to c_0(\Lambda)$ is injective.   

{\em \underline{Dense Range:}} Now to see that $T_{\bar{\theta}} : c_0(\Lambda) \to c_0(\Lambda)$ has dense range, one merely applies Lemma $\ref{lem:Adjoint}$. We may equivalently write $T_{\bar{\theta}}$ as an infinite matrix, say $B$, indexed by the elements of $\Lambda \times \Lambda$ such that
	\begin{equation}
		B = (b_{(i_1,j_1),(i_2,j_2)})_{\Lambda\times\Lambda} = \left\{
    		 \begin{array}{ll}
       		-\sum_{(i_1',j_1')} \cos(\bar{\theta}_{i_1',j_1'} - \bar{\theta}_{i_1,j_1}) & : (i_1,j_1) = (i_2,j_2)\\
       		\cos(\bar{\theta}_{i_2,j_2} - \bar{\theta}_{i_1,j_1}) & : (i_1,j_1)\sim (i_2,j_2) \\
       		0 & : otherwise
     		\end{array}
  		\right.
	\end{equation} 
 where we have used the notation $ (i_1,j_1)\sim (i_2,j_2)$ to denote that $(i_1,j_1)$ and $(i_2,j_2)$ are nearest-neighbours. Now again using the even symmetry of the cosine function we see that this matrix is clearly symmetric since for $(i_1,j_1) \neq (i_2,j_2)$ we have
\begin{equation}
	\begin{split}
	b_{(i_1,j_1),(i_2,j_2)} &= \left\{  
		\begin{array}{ll}
        		\cos(\bar{\theta}_{i_2,j_2} - \bar{\theta}_{i_1,j_1}) & : (i_1,j_1)\sim (i_2,j_2) \\
       		0 & : otherwise
     		\end{array}
  		\right.  \\
		&= \left\{ 
		\begin{array}{ll}
        		\cos(\bar{\theta}_{i_1,j_1} - \bar{\theta}_{i_2,j_2}) & : (i_1,j_1)\sim (i_2,j_2) \\
       		0 & : otherwise
     		\end{array}
  		\right.  \\
		&= b_{(i_2,j_2),(i_1,j_1)}. 
	\end{split}
\end{equation} 
This implies that $T_{\bar{\theta}}^* : \ell^1(\Lambda) \to \ell^1(\Lambda)$ merely acts as $T_{\bar{\theta}}$. Repeating the above arguments we see that $T_{\bar{\theta}}^*$ must be injective since the constant sequences are not in $\ell^1(\Lambda)$. This therefore gives that $T_{\bar{\theta}} : c_0(\Lambda) \to c_0(\Lambda)$ has dense range.
\end{proof} %End of proof of Operator T Lemma

\begin{cor} \label{lem:PsiPartial} %Partial derivative of G^2
	The operator $M_{33}:c_0(\Lambda) \to c_0(\Lambda)$ is injective and has dense range,
\end{cor}

\begin{proof}
	{\em \underline{Injectivity:}} Let us assume that $\psi = \{\psi_{i,j}\}_{(i,j)\in\Lambda} \in c_0(\Lambda)$ is such that $M_{33}\psi = 0$. This gives that for all $(i,j)\in\Lambda$ we have
	\begin{equation}
		i^{-1}\sum_{i',j'} \cos(\bar{\theta}_{i',j'} - \bar{\theta}_{i,j})(\psi_{i',j'} - \psi_{i,j}) = 0 \iff \sum_{i',j'} \cos(\bar{\theta}_{i',j'} - \bar{\theta}_{i,j})(\psi_{i',j'} - \psi_{i,j}) = 0.
	\end{equation} 
	since $i^{-1} > 0$ for all $i \geq 1$. Therefore, $M_{33}\psi = 0$ if and only if $T_{\bar{\theta}}\psi = 0$, as defined in the statement of Lemma $\ref{lem:OperatorT}$. Since $T_{\bar{\theta}}$ was shown to be injective, it follows that $\psi = 0$, thus showing that $M_{33}$ is injective.  
	
	{\em \underline{Dense Range:}} Let $y = \{y_{i,j}\}_{(i,j)\in\Lambda} \in c_0(\Lambda)$ and $\varepsilon > 0$ be arbitrary. Then, since $y \in c_0(\Lambda)$ there exists an $i_0 \geq 1$ such that $|y_{i,j}| < \varepsilon/2$ for all $i > i_0$. Define the element $\tilde{y}  = \{\tilde{y}_{i,j}\}_{(i,j)\in\Lambda} \in c_0(\Lambda)$ by
	\begin{equation}
		\tilde{y}_{i,j} = \begin{cases}
			y_{i,j} & \text{ if } i \leq i_0 \\
			0 & \text{ if } i > i_0
		\end{cases}.
	\end{equation}
	Then by construction we have that
	\begin{equation}
		\|y - \tilde{y}\|_\infty < \frac{\varepsilon}{2}.
	\end{equation}
	Now consider the element $y' = \{y'_{i,j}\}_{(i,j)\in\Lambda} \in c_0(\Lambda)$ defined by $y'_{i,j} = i\cdot\tilde{y}_{i,j}$ for all $(i,j)\in\Lambda$. Note that since $y'$ only has finitely many nonzero components, it therefore trivially belongs to $c_0(\Lambda)$. Then since $T_{\bar{\theta}}:c_0(\Lambda) \to c_0(\Lambda)$ has dense range, there exists a $\psi \in c_0(\Lambda)$ such that 
	\begin{equation}
		\|T_{\bar{\theta}}\psi - y'\|_\infty = \sup_{(i,j)\in\Lambda} \bigg|\sum_{i',j'}\cos(\bar{\theta}_{i',j'} - \bar{\theta}_{i,j})(\psi_{i',j'} - \psi_{i,j}) - y'_{i,j}\bigg| < \frac{\varepsilon}{2}.
	\end{equation} 
	Hence, 
	\begin{equation}
		\begin{split}
			\|M_{33}\psi - \tilde{y}\|_\infty &= \sup_{(i,j)\in\Lambda} \bigg|i^{-1}\sum_{i',j'}\cos(\bar{\theta}_{i',j'} - \bar{\theta}_{i,j})(\psi_{i',j'} - \psi_{i,j}) - \tilde{y}_{i,j}\bigg| \\
			&=  \sup_{(i,j)\in\Lambda} i^{-1}\bigg|\sum_{i',j'}\cos(\bar{\theta}_{i',j'} - \bar{\theta}_{i,j})(\psi_{i',j'} - \psi_{i,j}) - y'_{i,j}\bigg| \\
			&\leq \sup_{(i,j)\in\Lambda} \bigg|\sum_{i',j'}\cos(\bar{\theta}_{i',j'} - \bar{\theta}_{i,j})(\psi_{i',j'} - \psi_{i,j}) - y'_{i,j}\bigg| \\
			&< \frac{\varepsilon}{2}. 	
		\end{split}	
	\end{equation}
	Therefore, putting this all together gives
	\begin{equation}
		\|M_{33}\psi - y\|_\infty \leq \|M_{33}\psi - \tilde{y}\|_\infty + \|y - \tilde{y}\|_\infty < \frac{\varepsilon}{2} + \frac{\varepsilon}{2} = \varepsilon, 
	\end{equation}
	showing that $M_{33}$ has dense range since $\varepsilon > 0$ and $y \in c_0(\Lambda)$ were arbitrary. 
\end{proof} %End of Proof Partial Derivative D_\psi

One notices from Proposition $\ref{prop:FrechetDerivative}$ that $M$ is a block lower-triangular matrix. This greatly simplifies much of our analysis and implies that many of the properties of the diagonal elements carry over to the full block matrix $M$. In particular, Lemma $\ref{lem:SPartial}$ and Corollary $\ref{lem:PsiPartial}$ lead to the following corollary regarding the density of the range of the operator $M$ in $(\ref{DerivativeMatrix})$.

\begin{cor} \label{cor:MMatrix} %Corollary: M has dense range
	$M: X \to X$ is injective and has dense range. 
\end{cor}

\begin{proof}
	{\em \underline{Injectivity:}} Let us assume that $Mx = 0$, where $x = (\alpha,s,\psi)$. This gives
	\begin{equation} \label{MatrixVectorMultiplication}
		  \begin{bmatrix}
		0 \\
		0\\
		0
		\end{bmatrix} =
		 \begin{bmatrix}
		1 & {\bf 0} & {\bf 0}\\
		M_{21} & M_{22}& {\bf 0} \\
		{\bf 0} & M_{32} & M_{33}
	\end{bmatrix} \cdot \begin{bmatrix}
		\alpha \\
		s\\
		\psi
	\end{bmatrix} = \begin{bmatrix}
		\alpha \\
		M_{21} \alpha + M_{22} s \\
		M_{32} s + M_{33}\psi
	\end{bmatrix}.	
	\end{equation}
	One can immediately see that this implies that $\alpha = 0$. This in turn yields  
	\begin{equation}
		0 = M_{21} \alpha + M_{22} s = M_{22} s, 	
	\end{equation}
	since $M_{21}0 = 0$. From Lemma $\ref{lem:SPartial}$ we have that $M_{22}$ is injective, thus giving that $s = 0$.
	
	Having $s = 0$ gives 
	\begin{equation}
		0 = M_{32} s + M_{33}\psi = M_{33}\psi,	
	\end{equation} 
	since $M_{32}0 = 0$. From Lemma $\ref{lem:PsiPartial}$ we have that $M_{33}$ is injective, thus giving that $\psi = 0$. Hence, $Mx = 0$ implies that $x = 0$, and therefore $M$ is injective.
	
	{\em \underline{Density of Range:}} Let $\varepsilon > 0$. We wish to show that for any $x = (\alpha,s,\psi) \in X$ that there exists $x_\varepsilon = (\alpha_\varepsilon,s_\varepsilon,\psi_\varepsilon) \in X$ such that 
	\begin{equation}
		\|Mx_\varepsilon - x\|_X < \varepsilon.
	\end{equation}
	Using $(\ref{MatrixVectorMultiplication})$ above we have 
	\begin{equation} \label{DenseMatrixVectorMultiplication}
		Mx_\varepsilon - x =
		\begin{bmatrix}
			\alpha_\varepsilon - \alpha \\
			M_{21} \alpha_\varepsilon + M_{22} s_\varepsilon - s \\
			M_{32} s_\varepsilon + M_{33}\psi_\varepsilon - \psi	
		\end{bmatrix}.
	\end{equation}
	Immediately one sees that we may take $\alpha_\varepsilon = \alpha$ to get that $|\alpha_\varepsilon - \alpha| = 0 < \varepsilon$. Then denoting the bounded inverse of $M_{22}$ by $M_{22}^{-1}:\ell^\infty(\Lambda) \to \ell^\infty(\Lambda)$, we will take $s_\varepsilon = M_{22}^{-1}[s - M_{21}\alpha_\varepsilon]$ to see that
	\begin{equation}
		\|M_{22} s_\varepsilon - (s - M_{21} \alpha )\|_\infty = 0 < \varepsilon.
	\end{equation}
	Then using the density of the range of $M_{33}$ we have that there exists a $\psi_\varepsilon \in c_0(\Lambda)$ such that 
	\begin{equation}
		\|M_{33} \psi_\varepsilon - (\psi - M_{32} s_\varepsilon)\|_\infty < \varepsilon,
	\end{equation}
	since $(\psi - M_{32} s_\varepsilon)\in c_0(\Lambda)$, by definition. Putting this all together gives that if $x_\varepsilon = (\alpha_\varepsilon, s_\varepsilon,\psi_\varepsilon)$ as above we obtain 
	\begin{equation}
		\begin{split}
			&\|Mx_\varepsilon - x\|_X \\ 
			&= \max\{|\alpha_\varepsilon - \alpha|,\|M_{21} \alpha_\varepsilon + M_{22} s_\varepsilon - s\|_\infty, \|M_{32} s_\varepsilon + M_{33}\psi_\varepsilon - \psi\|_\infty \}\\ 
			&< \varepsilon,
		\end{split} 
	\end{equation}
	completing the proof.
\end{proof} %End of proof of corollary

\begin{rmk}
It should be noted that the density of the range of $M:X\to X$ is all that we can conclude about the image of this operator. Following from the fact that $M_{33}$ is not necessarily surjective, we find that $M:X\to X$ is not necessarily surjective. In fact, it is easy to see that $M_{33}$ does have a nontrivial kernel when acting on $\ell^\infty(\Lambda)$, the second dual of $c_0(\Lambda)$, which in turn can be used to show that $M_{33}$ does not have a bounded inverse, but we refrain from doing so here. This in turn gives that $M:X \to X$ is not surjective, thus preventing traditional Implicit Function Theorem arguments being applied to the situation. This necessitates the application of a nonstandard Implicit Function Theorem such as Theorem $\ref{thm:CravenNashed}$ to obtain solutions to $G \in S$ for $\alpha > 0$. The great benefit of Theorem $\ref{thm:CravenNashed}$ is that it can be applied to both the situation when $M:X \to X$ is surjective and when it is not, with no change in the analysis here. 
\end{rmk}

We will now work to define an approximate right inverse. Let us fix $\mu \in (0,1)$. To begin, let $y \in X$ be an element of the image of $M$. Then there exists a unique $x \in X$ such that $Mx = y$. We will define $B_\mu y := x$. In this way one sees that 
\begin{equation}
	\|MB_\mu y - y\|_X = 0 \leq \mu \|y\|_X
\end{equation} 
for each $y$ in the image of $M$. For any $y' = (\alpha_y,s_y,\psi_y) \in X$ not in the image of $M$, we proceed in much the same way as the proof of Corollary $\ref{cor:MMatrix}$. That is, we will define $B_\mu y' = x' = (\alpha_x,s_x,\psi_x) \in X$ to be such that
\begin{equation}
	\begin{split}
		&\alpha_x := \alpha_y, \\
		&s_x := M_{22}^{-1}[s_y - M_{21}\alpha_y], 
	\end{split}
\end{equation}
and $\psi_x \in c_0(\Lambda)$ to be an element such that
\begin{equation}
	\|M_{33} \psi_x - (\psi_y - M_{32} s_x)\|_\infty \leq \mu \|y\|_X,
\end{equation}
which is guaranteed to exist since $M_{33}$ has dense range. Hence, one sees that as in the proof of Corollary $\ref{cor:MMatrix}$, we have that 
\begin{equation}
	\|Mx' - y'\|_X \leq \mu \|y'\|_X.
\end{equation} 
Therefore the mapping $B_\mu : X \to X$ satisfies the first condition for an approximate right inverse. The reader should note that our mapping $B_\mu: X \to X$ is neither linear nor injective nor unique for any $\mu \in (0,1)$ since $M:X \to X$ is not surjective. 

We now require the appropriate space to bound our approximate right inverse $B_\mu$. To begin, recall that $c_0(\Lambda)$ exhibits a Shauder basis given by the canonical basis, $\{\delta_{i,j}\}_{(i,j)\in\Lambda} \subset c_0(\Lambda)$, where the element $\delta_{i,j}$ has zeros at every index but for a $1$ at the index $(i,j)$. Then writing
\begin{equation}
	\delta^\psi_{i,j} = (0,0,\delta_{i,j})\in X,
\end{equation}
we can uniquely write any $x = (0, 0,\psi)\in X$ as 
\begin{equation}
	x = \sum_{(i,j)\in\Lambda} \psi_{i,j}\delta^\psi_{i,j}.
\end{equation}
For any $n \geq 1$, consider the finite-dimensional closed subspace of $X$ given by
\begin{equation}
	E_n = \{(0,0,\psi) \in X|\ \psi_{i,j} = 0\ \forall i > n\}.
\end{equation}
One can see that for each $n \geq 1$ the subspace $E_n$ is merely those elements of $X$ in which $\alpha = 0$, $s = 0$ and the only nonzero components of $\psi = \{\psi_{i,j}\}_{(i,j)\in\Lambda}$ belong to the first $n$ columns of $\Lambda$. Furthermore, the set of vectors in $X$ given by 
\begin{equation} \label{BasisOfEn}
	B_n := \{\delta^\psi_{i,j}\}_{i\leq n} 
\end{equation}
forms a basis of $E_n$, for each $n \geq 1$.

\begin{lem} \label{lem:FiniteInverse} %Lemma: Finite dimensional inverse bound 
	For each $n \geq 1$, $M: E_n \to X$ has a bounded inverse from its range to $E_n$.	
\end{lem}

\begin{proof}
	Fix $n \geq 1$ and finite. Then $E_n$ is a finite-dimensional subspace and the image of the space under the mapping $M$, denoted $M(E_n)$, is a subspace of $X$. From the Rank-Nullity Theorem one has that $M(E_n)$ is finite-dimensional since $E_n$ is finite-dimensional. Hence, $M:E_n \to M(E_n)$ is an operator acting between finite-dimensional spaces. Since invertibility is equivalent to injectivity in finite-dimensions, $M:E_n \to M(E_n)$ has a bounded inverse since we have shown in Corollary $\ref{cor:MMatrix}$ that it is injective. 
\end{proof} %End of proof of finite dimensional inverse bound

Denoting $M(E_n)$ to be the image of $E_n$ under the action of $M$, we find that $M(E_n)$ is a finite-dimensional subspace of $X$ and therefore is closed. Moreover, since $M:E_n \to M(E_n)$ is invertible, it follows that the set $M(B_n) = \{M\delta^\psi_{i,j}\}_{i\leq n}$ is a linearly independent spanning set of $M(E_n)$. A well-known consequence of the Hahn-Banach theorem is that for each element of the spanning set $M(B_n)$, say $v$, there exists a linear functional $\phi_v \in X^*$ such that $\phi_v(v) = 1$ and $\phi_v(v') = 0$ for all $v' \in M(B_n)\setminus\{v\}$. This gives a natural continuous projection, $P_{M(E_n)}: X \to M(E_n)$ by 
\begin{equation} \label{ImageProjectionDefinition}
	P_{M(E_n)}x = \sum_{i\leq n} \phi^\psi_{i,j}(x)M\delta^\psi_{i,j},		
\end{equation}  
where $\phi^\psi_{i,j} \in X^*$ is the linear functional corresponding as above to the basis element $M\delta^\psi_{i,j} \in M(B_n)$. Linearity and continuity immediately follow from the fact that the elements of $X^*$ used to define $P_{M(E_n)}$ are linear and continuous.  

Now, since $P_{M(E_n)}: X \to M(E_n)$ is continuous for all $n\geq 1$, this implies that we may write $X = M(E_n) \oplus V_n$, where $V_n$ is the kernel of $P_M(E_n)$, which is closed. Let us denote $M(X)$ to be the image of $X$ under the action of $M:X \to X$. Consider the linear subspaces of $X$:
\begin{equation}
	U_n := M^{-1}(V_n \cap M(X)).
\end{equation}
That is, $U_n$ is the preimage of the complement of $M(E_n)$ in $M(X)$.

\begin{claim} \label{Claim1}
	For each $n \geq 1$ we have $E_n \cap U_n = \{0\}$.
\end{claim}

\begin{proof}
	Let us fix $n \geq 1$ and assume that $x \in E_n \cap U_n$. Then, since $x\in E_n$ we get that $Mx \in M(E_n)$. Similarly, since $x\in U_n$, this implies that $Mx \in V_n \cap M(X) \subseteq V_n$. So, $Mx \in M(E_n) \cap V_n$. But $M(E_n) \cap V_n = \{0\}$ by definition, giving that $Mx = 0$. Injectivity of $M$ implies that $x = 0$. 
\end{proof}

\begin{claim} \label{Claim2}
	$U_n$ is closed as a subspace of $X$ for every $n \geq 1$.
\end{claim}

\begin{proof}
	Fix $n \geq 1$ and let $\{x_k\}_{k=1}^\infty \subset U_n$ be a sequence converging in the norm of $X$ to some $x \in X$. We wish to show that $x \in U_n$. 
	
	For each $k\geq 1$, let us define $y_k := Mx_k$. By definition we have that $y_k \in V_n \cap M(X) \subset V_n$. Let us further define $y := Mx$ to see that
	\begin{equation}
		\|y_k - y\|_X = \|Mx_k - Mx\|_X \leq \|M\|_{op}\|x_k - x\|_X, 
	\end{equation}  	
	where $ \|M\|_{op}<\infty$ is the operator norm of $M:X \to X$. Then since $\|x_k - x\|_X \to 0$ as $k \to \infty$ it follows that $y_k \to y$ as $k \to \infty$ in the norm of $X$. Since $V_n$ is closed, we have that $y \in V_n$. Furthermore, $y = Mx \in M(X)$, so we see that $y \in V_n \cap M(X)$. Therefore, $x = M^{-1}y \in M^{-1}(V_n \cap M(X)) = U_n$, giving that $U_n$ is closed.  
\end{proof}

\begin{cor} \label{cor:XSplitting}
	For each $n\geq 1$ we have $X = E_n \oplus U_n$.
\end{cor}

\begin{proof}
	From Claims $\ref{Claim1}$ and $\ref{Claim2}$, we have that for every $n \geq 1$, the direct sum $E_n \oplus U_n$ is well-defined since both components are closed and mutually disjoint as subspaces of $X$. Furthermore, $E_n \oplus U_n \subseteq X$ by definition. It therefore only remains to show that $X \subseteq E_n \oplus U_n$. 
	
	Now fix $n \geq 1$ and take $x\in X$. Then $Mx \in M(X) \subset X = M(E_n) \oplus V_n$. Hence, there exists $y_E \in M(E_n)$ and $y_V \in V_n$ such that $Mx = y_E + y_V$. Now since $y_E \in M(E_n)$ and $M: E_n \to M(E_n)$ is invertible, there exists an unique $x_E \in E_n$ such that $Mx_E = y_E$. So,
	\begin{equation}
		M(x - x_E) = Mx - y_E = y_V,
	\end{equation} 	
	and therefore $y_V$ belongs to $M(X)$. Hence, $y_V \in V_n \cap M(X)$, which implies there exists an element $x_U \in U_n$ such that $Mx_U = y_V$. Putting this all together shows that 
	\begin{equation}
		M(x - x_E - x_U) = y - y_E - y_V = 0.
	\end{equation}
	From injectivity of $M$ we get that $x = x_E + x_U \in E_n \oplus U_n$, completing the proof. 
\end{proof}

Following the result of Corollary $\ref{cor:XSplitting}$ we see that for every $n \geq 1$ we have two decompositions of $X$ as the direct sum of closed subspaces: $E_n \oplus U_n$ and $M(E_n) \oplus V_n$. Moreover, one sees that $M(U_n) \subset V_n$, showing that $M$ preserves this splitting. Now, since $X = E_n \oplus U_n$, there exists a continuous projection $P_n: X \to E_n$ which acts on every $x = x_E + x_U \in E_n \oplus U_n$ by $P_n x = x_E$. By definition we have that the range of $P_n$ is given by $E_n$ and its kernel is exactly $U_n$.

\begin{claim} \label{Claim3} %Commutative Projections Claim
	For every $x\in X$ and $n \geq 1$, $MP_n x = P_{M(E_n)}Mx$.
\end{claim}

\begin{proof}
	Let $x \in X$ and $n \geq 1$. Then we may write 
	\begin{equation}
		x = P_nx + (I-P_n)x,
	\end{equation}
	where $I:X\to X$ is the identity operator on $X$, so that $P_nx \in E_n$ and $(I - P_n)x \in U_n$. Applying $M$ to this equation gives 
	\begin{equation}
		Mx = MP_nx + M(I - P_n)x,
	\end{equation}
	so that by definition we have $MP_nx \in M(E_n)$ and $M(I - P_n)x \in V_n$. Applying $P_{M(E_n)}$ now gives that
	\begin{equation}
		P_{M(E_n)}Mx = MP_nx, 	
	\end{equation} 
	since $V_n$ is the kernel of $P_{M(E_n)}$ by definition. This completes the proof of the claim. 
\end{proof}

Claim $\ref{Claim3}$ gives the following commutative diagram:

\begin{center}
\large
$\begin{array}[c]{ccc}
E_n\oplus U_n&\xrightarrow{\ \ M\ \ }&M(E_n)\oplus V_n\\
\bigg\downarrow\scriptstyle{P_n}&&\bigg\downarrow\scriptstyle{P_{M(E_n)}}\\
E_n \ \ &\xrightarrow{\ \ M\ \ }&M(E_n)
\end{array}$
\end{center}

Let us denote $M_n^{-1}:M(E_n) \to E_n$ the bounded inverse of $M:E_n \to M(E_n)$ for each $n \geq 1$. Moreover, let us write $C(n) := \|M_n^{-1}\|_{op}$, which is finite for each finite $n \geq 1$. One can see that $C(n)$ is an increasing function of $n$ such that $C(n) \to \infty$ as $n \to \infty$. This follows from the fact that $E_n \subseteq E_{n'}$ for each $1 \leq n \leq n'$ and that as $n \to \infty$ we are approaching an inverse of $M_{33}$, which cannot be bounded since $M:X \to X$ is not surjective.

\begin{lem} \label{lem:EProjection}
	For each $n \geq 1$ and $x = (\alpha,s,\psi)\in X$, we have that
	\begin{equation}
		P_n x = \sum_{i \leq n} \psi_{i,j}\delta^\psi_{i,j}.
	\end{equation}
\end{lem}

\begin{proof}
	Fix $n \geq 1$ and consider some $x = (\alpha,s,\psi)\in X$. Then $y = Mx \in M(X) \subset X$. Using $(\ref{ImageProjectionDefinition})$ gives 
	\begin{equation}
		P_{M(E_n)}y = \sum_{i\leq n} \phi^\psi_{i,j}(y)M\delta^\psi_{i,j}. 
	\end{equation}  
	Since $P_{M(E_n)}y \in M(E_n)$ we may apply $M_n^{-1}$ to see that 
	\begin{equation}
		M_n^{-1}P_{M(E_n)}y = M_n^{-1}MP_n x = P_n x,
	\end{equation}
	where we have applied the identity from Claim $\ref{Claim3}$. Then linearity of $M_n^{-1}$ gives that
	\begin{equation}
		P_n x = M_n^{-1}P_{M(E_n)}y = \sum_{i\leq n} \phi^\psi_{i,j}(y)\delta^\psi_{i,j}.
	\end{equation} 
	Since the set $B_n$ is a linearly independent spanning set of $E_n$, it follows that
	\begin{equation}
		\phi^\psi_{i,j}(y) = \psi_{i,j}		
	\end{equation}
	for all $(i,j)$. This proves the lemma. 
\end{proof}

We are now in a position to define an appropriate norm to bound our approximate right inverse. 

\begin{prop} \label{prop:NewNorm} %Proposition: New Norm
	Let $n \geq 1$ and define the function $\|\cdot \|_n : X \to [0,\infty)$ by
	\begin{equation} \label{nNorm}
		\|x\|_n = \|(\alpha,s,\psi)\|_n := \max\bigg\{|\alpha|, \|s\|_\infty, \|P_n x\|_X, \frac{1}{\|M\|_{op}}\|Mx\|_X\bigg\}
	\end{equation}
	for all $x\in X$, and $\|M\|_{op}$ denotes the operator norm of $M:X\to X$. Then $\|\cdot \|_n$ defines a norm on $X$. Moreover, for each $x \in X$ we have that $\|x\|_n \leq \|x\|_{n'} \leq \|x\|_X$ for any $1\leq n \leq n'$ and $\|x\|_n \to \|x\|_X$ as $n \to \infty$.   
\end{prop}

\begin{rmk}
	For the duration of this section we will let $X_n$ denote the completion of $X$ with respect to the norm $\|\cdot\|_n$.
\end{rmk}

\begin{proof} 
	To see that $(\ref{nNorm})$ defines a norm is a straightforward checking of the axioms which primarily follows from the injectivity of $M$ along with the linearity of both $P_n$ and $M$. 
	
	Now to see that such an ordering of the norms holds one observes that for any $n \geq 1$ and $x = (\alpha,s,\psi) \in X$, we have from Lemma $\ref{lem:EProjection}$ that
	\begin{equation}
		\|P_n x\|_X = \max_{(i,j)\in\Lambda,\ i\leq n} |\psi_{i,j}|.	
	\end{equation} 
	Hence, one can see that $\|P_n x\|_X \leq \|P_{n'}x\|_X$ for any $1\leq n \leq n'$ and $x \in X$. This also leads to the fact that $\|P_nx\|_X \leq \|x\|_X$ for all $n \geq 1$. Hence,
	\begin{equation}
		\begin{split}
		\|x\|_n &= \max\bigg\{|\alpha|, \|s\|_\infty, \|P_n x\|_X, \frac{1}{\|M\|_{op}}\|Mx\|_X\bigg\}\\
			 &\leq  \max\bigg\{|\alpha|, \|s\|_\infty, \|P_{n'} x\|_X, \frac{1}{\|M\|_{op}}\|Mx\|_X\bigg\}\\ 
			 &= \|x\|_{n'} 
		\end{split}
	\end{equation} 
	and since $\|Mx\|_X \leq \|M\|_{op}\|x\|_X$ we clearly have that $\|x\|_n \leq \|x\|_X$.
	
	Finally, if $x=(\alpha,s,\psi) \in X$, then $\psi \in c_0(\Lambda)$, implying that there exists $(i_0,j_0)\in\Lambda$ such that $|\psi_{i_0,j_0}| = \|\psi\|_\infty$. Then for every $n \geq i_0$ we have
	\begin{equation}
		\|P_n x\|_X = \max_{(i,j)\in\Lambda,\ i\leq n} |\psi_{i,j}| = \|\psi\|_\infty,		
	\end{equation}
	thus implying that
	\begin{equation}
		\|x\|_n = \max\bigg\{|\alpha|, \|s\|_\infty, \|\psi\|_\infty, \frac{1}{\|M\|_{op}}\|Mx\|_X\bigg\} \geq \|x\|_X. 
	\end{equation}
	Hence, for every $n \geq i_0$ we have $\|x\|_n = \|x\|_X$, giving convergence as $n \to \infty$. 
\end{proof} %End of proof of new norm

\begin{lem} \label{lem:InverseBounds} %Lemma: Bound on Approximate Right Inverse
	For every $\mu \in (0,1)$ and $n \geq 1$ there exists a bound $\Gamma(n)$, independent of $\mu$, such that $\|B_\mu y\|_n \leq \Gamma(n) \|y\|_X$ for every $y \in X$. Moreover, $\Gamma(n)$ is an increasing function of $n$ such that $\Gamma(n) \to \infty$ as $n \to \infty$. 
\end{lem} 

\begin{proof}
	Fix $\mu \in (0,1)$ and $n \geq 1$. Let us begin by considering $y = (\alpha_y,s_y,\psi_y)\in M(X)$. Then there exists a unique $x = (\alpha_x,s_x,\psi_x)\in X$ such that $Mx = y$. Immedately from the identity stated in Claim $\ref{Claim3}$ we have that
	\begin{equation}
		MP_n x = P_{M(E_n)} y
	\end{equation}
	and since $MP_n x \in M(E_n)$, we may apply $M_n^{-1}: M(E_n) \to E_n$ to find that
	\begin{equation}
		P_n x = M_n^{-1}MP_n x = M_n^{-1}P_{M(E_n)}y. 
	\end{equation}
	Hence, 
	\begin{equation}
		\|P_n x\|_X = \|M_n^{-1}P_{M(E_n)}y\|_X \leq C(n) \|P_{M(E_n)}y\|_X \leq C(n)\cdot \|P_{M(E_n)}\|_{op} \|y\|_X,	
	\end{equation}
	where $\|P_{M(E_n)}\|_{op} < \infty$ denotes the operator norm of the bounded projection $P_{M(E_n)}: X \to M(E_n)$. 
	
	Now, since $Mx = y$, we may use the form of $M$ from Proposition $\ref{prop:FrechetDerivative}$ to find that we necessarily have
	\begin{equation}
		\begin{split}
			&\alpha_x := \alpha_y, \\
			&s_x := M_{22}^{-1}[s_y - M_{21}\alpha_y], 
		\end{split}
	\end{equation}
	where $M_{22}^{-1}:\ell^\infty(\Lambda) \to \ell^\infty(\Lambda)$ denotes the bounded inverse of $M_{22}$. Therefore, 
	\begin{equation}
		|\alpha_x| = |\alpha_y| \leq \max\{|\alpha_y|, \|s_y\|_\infty, \|\psi_y\|_\infty\} = \|y\|_X	
	\end{equation}
	and
	\begin{equation}
		\|s_x\|_\infty \leq \|M_{22}^{-1}\|_{op}[\|s_y\|_\infty + \|M_{21}\|_{op}|\alpha_y|] \leq  \|M_{22}^{-1}\|_{op}[1 + \|M_{21}\|_{op}] \|y\|_X, 
	\end{equation}
	where all operator norms are finite.
	
	Putting this all together shows that for any $y = Mx \in M(X)$ we get that $B_\mu y = x$ and
	\begin{equation}
		\begin{split}
			\|B_\mu y\|_n &= \|x\|_n \\
			&= \max\{|\alpha_x|, \|s_x\|_\infty, \|P_nx\|_X, \frac{1}{\|M\|_{op}} \|y\|_X\} \\
			&\leq \max\bigg\{1,\|M_{22}^{-1}\|_{op}[1 + \|M_{21}\|_{op}], C(n)\cdot \|P_{M(E_n)}\|_{op}, \frac{1}{\|M\|_{op}}\bigg\} \|y\|_X.
		\end{split}
	\end{equation}
		
	Now for any $y \in X\setminus M(X)$, recall that we have defined $x := B_\mu y \in X$ to be so that
	\begin{equation} \label{CloseToY}
		\|Mx - y\|_X \leq \mu \|y\|_X.	
	\end{equation}
	But since $y' := Mx$ belongs to the image of $M$, it follows that $B_\mu y' = x$ as well. Hence, from above we get
	\begin{equation}
		\|B_\mu y'\|_n \leq \max\{1,\|M_{22}^{-1}\|_{op}[1 + \|M_{21}\|_{op}],C(n)\cdot\|P_{M(E_n)}\|_{op},\frac{1}{\|M\|_{op}}\} \|y'\|_X. 
	\end{equation}
	Then rearranging $(\ref{CloseToY})$ with the reverse triangle inequality gives that
	\begin{equation}
		\|y'\|_X = \|Mx\|_X \leq (1 + \mu)\|y\|_X.
	\end{equation}
	Therefore, putting all of this together gives 
	\begin{equation}
		\begin{split}
			\|B_\mu y\|_n &= \|B_\mu y'\|_n \\
			&\leq \max\bigg\{1,\|M_{22}^{-1}\|_{op}[1 + \|M_{21}\|_{op}], C(n)\cdot \|P_{M(E_n)}\|_{op}, \frac{1}{\|M\|_{op}}\bigg\} \|y'\|_X \\
			&\leq (1 + \mu)\cdot\max\bigg\{1,\|M_{22}^{-1}\|_{op}[1 + \|M_{21}\|_{op}], C(n)\cdot \|P_{M(E_n)}\|_{op}, \frac{1}{\|M\|_{op}}\bigg\}\|y\|_X \\
			&\leq 2\cdot\max\bigg\{1,\|M_{22}^{-1}\|_{op}[1 + \|M_{21}\|_{op}], C(n)\cdot \|P_{M(E_n)}\|_{op}, \frac{1}{\|M\|_{op}}\bigg\} \|y\|_X, 	
		\end{split}
	\end{equation}
	since $\mu < 1$. 
	
	Now, we have that $1 \leq \|P_{M(E_n)}\|_{op} < \infty$ for every $n$ so that $C(n)\cdot\|P_{M(E_n)}\|_{op} \to \infty$ as $n \to \infty$ from the properties of $C(n)$. Therefore, we will define
	\begin{equation} \label{GammaBound}
		\Gamma(n) := 2\cdot\max\bigg\{1,\|M_{22}^{-1}\|_{op}[1 + \|M_{21}\|_{op}], C(n)\cdot \|P_{M(E_n)}\|_{op}, \frac{1}{\|M\|_{op}}\bigg\}  
	\end{equation}
	for each $n \geq 1$ to find that $\Gamma(n) \to \infty$ as $n \to \infty$. We also have that $\|B_\mu y\|_n \leq \Gamma(n)\|y\|_X$ for all $\mu \in (0,1)$ and $y \in X$.
\end{proof} %End of proof of bound on approximate right inverse

Now that we have a bound on the norm of $B_\mu:X \to X$ as a function of $n$, we wish to translate this into an appropriate bound in $\mu$ to successfully apply Theorem $\ref{thm:CravenNashed}$. That is, we wish to obtain a function $n(\mu):(0,1) \to \mathbb{N}$ such that $\Gamma(n(\mu)) \leq k_0 \mu^{-\frac{1}{2}}$ for all $\mu \in (0,1)$. We obtain this through an inductive definition of the function $n(\mu)$. To begin, take $k_0 > 0$ such that
\begin{equation}
	\Gamma(1) \leq k_0.
\end{equation}
Then, let $\mu_1 \in (0,1)$ be such that 
\begin{equation}
	\Gamma(2) \leq k_0\mu_1^{-\frac{1}{2}}.
\end{equation}
One should note that this value $\mu_1$ can always be found since $k_0 \mu^{-\frac{1}{2}}$ is unbounded as $\mu \to 0^+$ and $\Gamma(n)$ is finite for every $n \geq 1$. Now we define $n(\mu) = 1$ on $(\mu_1,1)$. This clearly gives $\Gamma(n(\mu)) \leq k_0\mu^{-\frac{1}{2}}$ on $(\mu_1,1)$. Next, let $\mu_2 \in (0,\mu_1)$ such that 
\begin{equation}
	\Gamma(3) \leq k_0\mu_2^{-\frac{1}{2}}.
\end{equation}
Again, this can be found since $k_0 \mu^{-\frac{1}{2}}$ is unbounded as $\mu \to 0^+$. We define $n(\mu) = 2$ on $(\mu_2,\mu_1]$, now giving that $\Gamma(n(\mu)) \leq k_0\mu^{\frac{1}{2}}$ on $(\mu_2,1)$. 

This process continues inductively in that for any $n = d \geq 2$, we define $\mu_{d} \in (0,\mu_{d-1})$ so that
\begin{equation}
	\Gamma(d+1) \leq k_0\mu_d^{-\frac{1}{2}}.	
\end{equation}  
Then $n(\mu) = d$ on $\mu \in (\mu_d,\mu_{d-1}]$ so that $\Gamma(n(\mu)) \leq k_0\mu^{-\frac{1}{2}}$ on $(\mu_d,1)$. This process is illustrated in Figure $\ref{fig:MuSequence}$.

\begin{figure} %Figure: Mu Sequence
	\centering
	\includegraphics[width = 7cm]{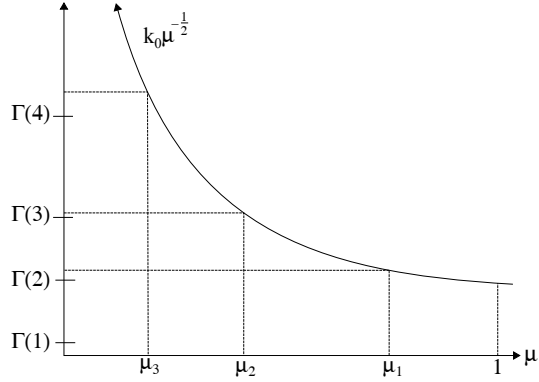}
	\caption{Visualization of how we inductively define $n(\mu)$ on $(0,1)$.}
	\label{fig:MuSequence}
\end{figure} 

It follows from the fact that $C(n) \to \infty$ as $n \to \infty$ that we have $\mu_d \to 0$ as $d \to \infty$. Hence, we have a function $n(\mu)$ which gives us that for any $\mu \in (0,1)$
\begin{equation}
	\|MB_\mu y - y\|_X \leq \mu \|y\|_X \ {\rm and}\ \|B_\mu y\|_{n(\mu)} \leq k_0\mu^{-\frac{1}{2}} \|y\|_X
\end{equation}
for every $y \in X$. Therefore, $B_\mu: X \to X$ is an approximate right inverse of $M$ for each $\mu \in (0,1)$ satisfying the hypothesis of Theorem $\ref{thm:CravenNashed}$.

The final step to applying Theorem $\ref{thm:CravenNashed}$ to our situation is determining the existence of the element $c$ such that $-[G(0,0,0) + Mc] \in S$ and $\|c\|_{n(\mu)} = 1$. One should note that by definition $X$ is a dense subset of $X_n$ for every $n \geq 1$. Then we extend $M:X \to X$ to a linear operator $\overline{M}^n : X_n \to X$ with the following lemma.

\begin{lem}
	For each $n\geq 1$, $M:X \to X$ extends to a unique bounded linear operator $\overline{M}^n : X_n \to X$ such that $\overline{M}^n|_X = M$.  
\end{lem}

\begin{proof}
	By definition $X$ is a dense subset of $X_n$, for every $n \geq 1$. Then for any $n \geq 1$ and an $x \in X \subset X_n$ we have 
	\begin{equation}
		\begin{split}
			\|Mx\|_X &= \frac{\|M\|_{op}}{\|M\|_{op}}\|Mx\|_X \\
			&\leq \|M\|_{op}\cdot\max\{|\alpha|,||s||_\infty,\|P_nx\|_X,\frac{1}{\|M\|_{op}}\|Mx\|_X\}\\ 
			&= \|M\|_{op}\|x\|_n. 
		\end{split}
	\end{equation} 
	That is, $M:X \subset X_n \to X$ is a densely defined, bounded linear operator. The existence of a unique bounded linear extension, denoted $\overline{M}^n : X_n \to X$, directly follows from the Bounded Linear Transformation (B.L.T.) Theorem, a proof of which can be found in $\cite{ReedSimon}$. 
\end{proof}

The operator $\overline{M}^n$ is often referred to as the closure of the operator $M$ with respect to the space $X_n$. For ease of notation we will simply write $M:X_n \to X$ to denote this closure with respect to the space $X_n$, although one should always keep in mind that the closure entirely depends on the value of $n$. We now prove the following lemma.

\begin{lem} \label{lem:Surjective} %Lemma: Surjective operator to find c's
	$M: X_n \to X$ is surjective for every $n \geq 1$. 
\end{lem}

\begin{proof}
	To show that $M:X_n \to X$ is surjective, we show that its range is closed and dense. First, since $X \subset X_n$, from Corollary $\ref{cor:MMatrix}$ we have that $M:X_n \to X$ has dense range. Now let $y \in X$. From the density of the range of $M:X \to X$, there exists a sequence $\{y_k\}_{k = 1}^\infty$ in the range of $M$ such that $\|y_k - y\|_X \to 0$ as $k \to \infty$. Then by definition, for each $y_k$, there exists $x_k \in X$ such that $Mx_k = y_k$. Notice that for any $\mu \in (0,1)$ we have $B_\mu y_k = x_k$ since $y_k$ belongs to the range of $M:X \to X$. Furthermore, since $y_k - y_{k'}$ belongs to the range of $M:X \to X$, we have that $B_\mu(y_k - y_{k'}) = x_k - x_{k'}$ for any $k,k'\geq 1$. Therefore, from Lemma $\ref{lem:InverseBounds}$ we have that 
	\begin{equation}
		\|x_k - x_{k'}\|_n = \|B_\mu(y_k - y_{k'})\|_n \leq \Gamma(n) \|y_k - y_{k'}\|_X.
	\end{equation}  
	Since $\Gamma(n) < \infty$ and $\{y_k\}_{k = 1}^\infty$ is a Cauchy sequence, it follows that $\{x_k\}_{k = 1}^\infty$ is a Cauchy sequence in $X_n$. Since $X_n$ is complete, there exists an element $x \in X_n$ such that $x_k \to x$ as $k \to \infty$. Then by the continuity of $M:X_n \to X$, it follows that $Mx = y$. Hence, $M:X_n \to X$ is closed.  
\end{proof} %End of proof of surjective operator

Finally, we are now in a position to apply Theorem $\ref{thm:CravenNashed}$ to our mapping $G$ in $(\ref{GMapping})$.

\begin{proof}[Proof of Theorem \ref{thm:ReducedSolution}]
Recall that our interest lies in the closed convex cone $S = \mathbb{R} \times \{0\} \times \{0\}$ defined in $(\ref{ConvexCone})$. For each $\mu \in (0,1)$ we have the relation $n(\mu)$ defined above, and an approximate right inverse $B_\mu:X \to X$ which has a bound function given by $k_0\mu^{-\frac{1}{2}}$, for some $k_0 > 0$, thus satisfying the hypothesis of Theorem $\ref{thm:CravenNashed}$. 

Then for any $\mu$ small and fixed, from Lemma $\ref{lem:Surjective}$ there exists a $c' \in X_{n(\mu)}$ such that 
\begin{equation}
	Mc' = \begin{pmatrix}
		1 \\
		0 \\
		0
	\end{pmatrix} \in S.
\end{equation} 
Taking $c = \frac{1}{\|c'\|_{n(\mu)}} c' \in X_{n(\mu)}$ gives $-Mc \in S$ and $\|c\|_{n(\mu)} = 1$. Hence, 
\begin{equation}
	-[G(0,0,0) + Mc] =  \frac{1}{\|c'\|_{n(\mu)}}\begin{pmatrix}
		-1 \\
		0 \\
		0
	\end{pmatrix} \in S,
\end{equation}
from $G(0,0,0) = (0,0,0)^T$ and the linearity of $M$. Therefore, from Theorem $\ref{thm:CravenNashed}$ we have that upon writing $c = (c_1,c_2,c_3)$, there exists a solution 
\begin{equation}
	\begin{pmatrix}
		\alpha(t) \\
		s(t) \\
		\psi(t)
	\end{pmatrix} = t\begin{pmatrix}
		c_1 \\
		c_2 \\
		c_3
	\end{pmatrix} + \begin{pmatrix}
	 \eta_{\alpha}(t) \\
	 \eta_s(t) \\
	 \eta_\psi(t) 
	 \end{pmatrix} \in X_{n(\mu)},	  
\end{equation}
to $-G(\alpha,s,\psi) \in S$ for some sufficiently small $t > 0$. This implies that both $G^1(\alpha,s,\psi) = 0$ and $G^2(\alpha,s,\psi) = 0$ at these solutions. 

A close inspection of the proof of Theorem $\ref{thm:CravenNashed}$ in $\cite{CravenNashed}$ shows that this solution is obtained by solving 
\begin{equation} \label{SolutionEquation}
	G(\alpha(t),s(t),\psi(t)) - G(0,0,0) - tMc = 0
\end{equation}
for $t > 0$. Then, using the fact that $Mc = \frac{1}{\|c'\|_{n(\mu)}}(1,0,0)^T$ and the form of $G$, we have that the first component of $(\ref{SolutionEquation})$ gives
\begin{equation}
	\alpha(t) = \frac{1}{\|c'\|_{n(\mu)}}t. 
\end{equation} 
Since, $\frac{1}{\|c'\|_{n(\mu)}} > 0$ we have that $\alpha(t) > 0$ for all values of $t > 0$ for which it is defined. Furthermore, $\eta_\alpha(t) \equiv 0$ for all $t > 0$ in its domain. The function $\alpha(t)$ is locally invertible allowing one to consider the function 
\begin{equation}
	t(\alpha) = \|c'\|_{n(\mu)} \alpha,
\end{equation}
for sufficiently small $\alpha > 0$. Therefore solutions $(\alpha(t),s(t),\psi(t))$ can be re-parametri\-zed in terms of small values of $\alpha > 0$. 

An a priori check reveals that since $s(t)$ (or equivalently $s(\alpha)$) represents the second component of an element in $X_{n(\mu)}$ and that $s(t) \to 0$ as $t \to 0^+$, we notice that for sufficiently small $t > 0$ we do indeed have $\|s(t)\|_\infty < \frac{a}{2}$, and therefore belonging to the domain of $G$. Furthermore, as long as this solution exists we have that $s(t) \in \ell^\infty(\Lambda)$ by definition of the norm on $X_{n(\mu)}$, thus showing that the radial perturbation $s$ is uniformly bounded and can be made uniformly as small as necessary. 

Then as noted at the beginning of this section, the zeros of $G^1$ and $G^2$ lie in one-to-one correspondence with those of $(\ref{FMapping})$. This gives solutions parametrized by sufficiently small $\alpha > 0$ given by
\begin{equation}
	\begin{pmatrix}
		r(\alpha) \\
		\theta(\alpha)
	\end{pmatrix} =\begin{pmatrix}
		{\bf a} \\
		\bar{\theta}
	\end{pmatrix} + t(\alpha)\begin{pmatrix}
		c_2 \\
		 c_3
	\end{pmatrix} + \begin{pmatrix}
		\eta_s(t(\alpha)) \\
		\eta_\psi(t(\alpha))
	\end{pmatrix}	
\end{equation}
to $(\ref{FMapping})$. Furthermore, letting $\alpha \to 0^+$ gives $t \to 0^+$, thus giving that $r(\alpha) \to {\bf a}$ and $\theta(\alpha) \to \bar{\theta}$, finishing the proof of Theorem $\ref{thm:ReducedSolution}$.
\end{proof}

\section{Discussion} \label{sec:Discussion} %Section: Discussion ---------------------------------------------------------------------------------------------------------------------

In this work we have successfully demonstrated the persistence of rotating wave solutions to system $(\ref{LambdaOmegaLDS})$ off of the anti-continuum limit $\alpha = 0$. This result was achieved by carefully constructing a mapping whose roots can be shown to lie in one-to-one correspondence with rotating wave solutions to the full Lambda-Omega system in question, and then applying a non-standard Implicit Function Theorem to this mapping. This work therefore has opened the door to a wide range of subsequent analyses, particularly in an effort to extend the scientific communities understanding of rotating waves in the absence of key symmetries in the underlying mathematical equations. Here we will discuss some avenues for future work which would be of great interest. 

First, our work has achieved a solution so that the phases around any concentric ring about the centre of rotation range over the values $0$ to $2\pi$, thus our rotating wave is of the single-armed variety. Although our work has been entirely focussed on single-armed rotating waves, it should be noted by the reader that $m$-armed rotating waves are defined in an analogous way in that it would have its phases ranging over the values $0$ up to $2\pi m$ around any concentric ring about the centre of rotation. Although $m$-armed rotating waves with $m > 1$ are not well-documented in the typical motivating application of cardiac electrophysiology, they are known to exist in nature nonetheless. Take for example the Belousov-Zhabotinskii chemical reactions, where spiral waves have been observed with up to six arms $\cite{MultiArmedSpiral}$. Such investigations could potentially motivate the study of multi-armed rotating waves in the discrete spatial setting as well. 

Unlike continuous spatial models where much of the analysis of single and multi-armed rotating waves can be handled in a similar way, the analysis undertaken here does not seem to lend itself to the existence of multi-armed rotating waves. That is, for a positive integer $m$, an $m$-armed rotating wave solution to $(\ref{LambdaOmegaLDS})$ would satisfy 
\begin{equation}
	z_{i,j}(t) = \bar{r}_{i,j}e^{\Omega(\alpha)t + m\bar{\theta}_{i,j}},
\end{equation}
where $(\bar{r}_{i,j},\bar{\theta}_{i,j})$ are steady-state solutions to 
\begin{equation} \label{MultiArmedPolar}
	\begin{aligned}
	\begin{split}
		\dot{r}_{i,j} = \alpha\sum_{i',j'}& [r_{i',j'}\cos(m(\theta_{i',j'} - \theta_{i,j})) - r_{i,j}] + r_{i,j}\lambda(r_{i,j}), \\
		\dot{\theta}_{i,j} = \alpha\sum_{i',j'}& \frac{r_{i',j'}}{r_{i,j}}\sin(m(\theta_{i',j'} - \theta_{i,j})) + \alpha\omega_1(r_{i,j},\alpha).
	\end{split}
	\end{aligned}
\end{equation}  
Despite the cases for $m > 1$ and $m = 1$ appearing quite similar, we remark that a steady-state solution to the phase components as $\alpha \to 0^+$ must satisfy
\begin{equation} \label{MultiArmedPhase}
	0 = \sum_{i',j'} \sin(m(\theta_{i',j'} - \theta_{i,j}))
\end{equation}
where we again use Hypothesis $\ref{Hyp:LambdaOmega}$ to have that $r_{i,j} \to a$ as $\alpha \to 0^+$. When $m > 1$ this coupling function does not satisfy the hypotheses of the work in the first part of this investigation and therefore requires potentially new techniques to obtain a solution. Even in the case when a solution can be obtained for a specific $m > 1$ we may not necessarily have 
\begin{equation}
	\cos(m(\theta_{i',j'} - \theta_{i,j})) \geq 0	
\end{equation}  
for all nearest-neighbour interactions. In this case the techniques used in the proof of Lemma $\ref{lem:PsiPartial}$ to show injectivity becomes significantly more complicated, if they can be undertaken at all. 

Numerical investigations undertaken in $\cite{Paullet}$ concluded that a $6\times 6$ lattice can exhibit a two-armed rotating wave solutions, thus leading one to conjecture that there exists an extension of this work to at least two-armed rotating waves. Furthermore, it is intuitive to think that spirals with $2$ and $4$ arms could at least be found in the system $(\ref{LambdaOmegaLDS})$ since these number of arms are commensurate with the symmetries of the lattice. Therefore, we could potentially work with a similar reduced phase system and exploit the symmetries of the square lattice to extend this solution over the whole lattice, much like what was done in $(\ref{SymmetryExtensions2})$. On the other hand, values such as $m = 3$ will be incommensurate with the symmetries of the lattice, thus presenting a technical hurdle in the analysis which at present has no straightforward way of being overcome. 
    
The second question that immediately comes to mind upon obtaining this rotating wave solution for small $\alpha > 0$ is: what is the stability of this solution? This is clearly a natural follow-up investigation which is certainly not as straightforward as one expects on first glance. That is, the failure for the linearization about the rotating wave at $\alpha = 0$ to be invertible immediately eliminates the possibility of the linearization to exhibit exponential stability, and therefore traditional dynamical systems techniques of extending linearized stability to local nonlinear stability is not possible. This then clearly necessitates a greater range of functional-analytical techniques and is expected to be the subject of a subsequent investigation.  

Finally, the primary motivation for this work was to investigate the effect that the loss of continuous Euclidean symmetry has upon the dynamics and bifurcations of rotating waves. At present this is still an open problem which is ambitious and long-term and therefore currently has ill-defined short-term goals. It is the intention of this work to help initiate formal analyses of the dynamics of these waves, and in particular, how they differ from the continuous spatial setting. Hence we have created a foundation in which further studies can expound upon in much the same way the seminal work of Zinner has done for traveling waves in the discrete spatial medium $\cite{Zinner}$.

\section*{Acknowledgements} %Acknowledgements ------------------------------------------------------------------------------------------------------------------------------------------------------------------------------------

This work was supported by an Ontario Graduate Scholarship while at the University of Ottawa. The author is very thankful to his supervisors Benoit Dionne and Victor LeBlanc for their guidance on all matters related to the work.

%Appendix ------------------------------------------------------------------------------------------------------------------------------------------------------------------------------------------------------------------------------------------ 
\appendix
\section{Strongly Fr\'echet Differentiable Implies Restricted Strongly Hadamard Differentiable} \label{sec:Appendix1} %Appendix: Differentiability Implications -------------------------------------------------------------

To begin, let $X$ and $Y$ be Banach spaces. The function $F:X\to Y$ is {\em strongly Hadamard differentiable} at the point $x_0 \in X$ if there exists a bounded linear operator $A:X\to Y$ such that, for any continuous function $u:[0,\infty) \to X$ for which $u'(0^+)$ exists and $u(0) = x_0$, the composition $F\circ u$ is strongly differentiable at $0^+$ with derivative $(F\circ u)'(0^+) = Au'(0^+)$. Furthermore, $F:X\to Y$ is called {\em restrictedly strongly Hadamard differentiable} at the point $x_0 \in X$ if the strongly Hadamard differentiable property holds when $u$ is restricted to being strongly differentiable at $0^+$. 

\begin{prop}
	Assume $F:X\to Y$ is strongly Fr\'echet differentiable at the point $x_0 \in X$. Then $F$ is restricted strongly Hadamard differentiable at $x_0$. 
\end{prop}     

\begin{proof}
	Let $\|\cdot\|_X$ and $\|\cdot\|_Y$ denote the norms of the Banach spaces $X$ and $Y$, respectively. Fix $\varepsilon > 0$. Then since $F: X \to Y$ strongly Fr\'echet differentiable at the point $x_0 \in X$, there exists a bounded linear operator $A:X \to Y$ and $\delta_1 > 0$  such that 
	\begin{equation} \label{Derivative1}
		\|F(x_1) - F(x_2) - A(x_1 - x_2)\|_Y < \varepsilon \|v_1 - v_2\|_X,
	\end{equation} 
	for all $\|x_1 - x_0\|_X, \|x_2 - x_0\|_X < \delta_1$.
	
	Now let $u:[0,\infty) \to X$ be a continuous function which is strongly differentiable at $0^+$ and satisfies $u(0) = x_0$. Then since $u$ is continuous and satisfies $u(0) = x_0$, there exists a $\delta_2 > 0$ such that $\|u(t) - x_0\|_V < \delta_1$ for all $t \in [0,\delta_2)$. Furthermore, since $u$ is strongly differentiable at $0^+$ there exists a $\delta_3 > 0$ such that 
	\begin{equation} \label{Derivative2}
		\|u(t_1) - u(t_2) - u'(0^+)(t_1 - t_2)\|_X < \varepsilon |t_1 - t_2|,
	\end{equation}    
	for all $0 < t_1,t_2 < \delta_3$.
	
	Now, let $\delta := \min\{\delta_2,\delta_3\} > 0$ and consider any $t_1,t_2 \in [0,\delta)$. Then, from $(\ref{Derivative1})$ and $(\ref{Derivative2})$ we have 
	\begin{equation}
		\begin{split}
		\begin{aligned}
		\|F(u(t_1)) - F(u(t_2)) - &A(u(t_1) - u(t_2))\|_Y \\
		&< \varepsilon \|u(t_1) - u(t_2)\|_X \\
		&\leq \varepsilon\|u(t_1) - u(t_2) - u'(0^+)(t_1 - t_2)\|_X + \varepsilon\|u'(0^+)(t_1 - t_2)\|_X \\
		&< \varepsilon^2 |t_1 - t_2| + \varepsilon|u'(0^+)\|t_1 - t_2|. 
		\end{aligned}
		\end{split}
	\end{equation}
	Similarly, using $(\ref{Derivative2})$ and the fact that $A$ is bounded linear, denoting its operator norm $\|A\|_{op}$, we obtain
	\begin{equation}
		\begin{split}
		\begin{aligned}
		\|Au(t_1) - Au(t_2) - Au'(0^+)(t_1 - t_2)\|_Y &\leq \|A\|_{op}\|u(t_1) - u(t_2) - u'(0^+)(t_1 - t_2)\|_X \\
		&< \varepsilon \|A\|_{op} |t_1 - t_2|. 
		\end{aligned}
		\end{split}   
	\end{equation}
	Putting this all together shows that 
	\begin{equation}
		\begin{split}
		\begin{aligned}
		\|F(u(t_1)) - F(u(t_2)) - &Au'(0^+)(t_1 - t_2)\|_Y  \\ 
		&\leq \|F(u(t_1)) - F(u(t_2)) - A(u(t_1) - u(t_2))\|_Y \\ 
		&\ \ \ \ \ + \|Au(t_1) - Au(t_2) - Au'(0^+)(t_1 - t_2)\|_Y \\
		&< \varepsilon^2 |t_1 - t_2| + \varepsilon|u'(0^+)\|t_1 - t_2| + \varepsilon \|A\|_{op} |t_1 - t_2|.  
		\end{aligned}
		\end{split} 	
	\end{equation}
	Dividing by $|t_1 - t_2|$ gives
	\begin{equation}
		\frac{\|F(u(t_1)) - F(u(t_2)) - Au'(0^+)(t_1 - t_2)\|_Y}{|t_1 - t_2|} < \varepsilon^2 + \varepsilon(\|A\|_{op} + |u'(0^+)|).	
	\end{equation}
	Since $\varepsilon > 0$ was arbitrary, this quotient can be made arbitrarily small, thus completing the proof. 
\end{proof}

\section{Proof of Proposition $\ref{prop:FrechetDerivative}$} \label{sec:Appendix2} %Appendix: Proof of partial derivative of M ---------------------------------------------------------------------------------------------------------------------

Here we will prove Proposition $\ref{prop:FrechetDerivative}$. This proof is broken down into two lemmas which lead to the proof of the Proposition. Recall that $B_{\frac{a}{2}}(0)$ denotes the ball of radius $\frac{a}{2}$ centred at $0$ in the space $\ell^\infty(\Lambda)$. 

\begin{lem} \label{lem:G1Derivative} %Derivative of G1
	$G^1(\alpha,s,\psi) : \mathbb{R} \times B_{\frac{a}{2}}(0) \times c_0(\Lambda) \to \ell^\infty(\Lambda)$ is strongly Fr\'echet differentiable at $(\alpha,s,\psi) = (0,0,0)$. The Fr\'echet derivative at this point is the bounded linear operator which acts by $(\alpha,s,\psi) \mapsto M_{21}\alpha + M_{22}s$, where $M_{21}$ and $M_{22}$ are as defined in $(\ref{M21})$ and $(\ref{M22})$, respectively.
\end{lem}

\begin{proof}
	For ease of notation let us return to the variables $r_{i,j} = a + s_{i,j}$ and $\theta_{i,j} = \bar{\theta}_{i,j} + \psi_{i,j}$ for all $(i,j) \in \Lambda$. We will write $r = \{r_{i,j}\}_{(i,j)\in\Lambda}$ and $\theta = \{\theta_{i,j}\}_{(i,j)\in\Lambda}$. Note that $G^1(\alpha,s,\psi)$ strongly Fr\'echet differentiable at $(\alpha, s,\psi) = (0,0,0)$ is equivalent to $G^1(\alpha,r,\theta)$ strongly Fr\'echet differentiable at $(\alpha,r,\theta) = (0, {\bf a}, \bar{\theta})$. 
	
	Begin by taking any $\varepsilon > 0$ and consider $(\alpha^1,s^1,\psi^1),(\alpha^2,s^2,\psi^2) \in \mathbb{R} \times B_{\frac{a}{2}}(0) \times c_0(\Lambda)$. We denote $r^1$ and $r^2$ to correspond to $s^1$ and $s^2$, respectively. Similarly $\theta^1$ and $\theta^2$ correspond to $\psi^1$ and $\psi ^2$. 
	
	Since consine is an infinitely differentiable function on $\mathbb{R}$, by applying Theorem $\ref{thm:StrongDiff}$ to the function $f:\mathbb{R}^4\to\mathbb{R}$ acting by $f(x_1,x_2,x_3,x_4) = x_1x_2\cos(x_3 - x_4)$ at the point $(x_1,x_2,x_3,x_4) = (0,a,\bar{\theta}_{i',j'},\bar{\theta}_{i,j})$ for each $(i,j)\in\Lambda$ and $(i',j')$ we have that there exists $\delta_1 > 0$ such that 
	\begin{equation} \label{CosineDiff}
		\begin{split}
		|\alpha^1r^1_{i',j'}\cos(\theta^1_{i',j'}& - \theta^1_{i,j}) - \alpha^2r^2_{i',j'}\cos(\theta^2_{i',j'} - \theta^2_{i,j}) - (\alpha^1 - \alpha^2)a\cos(\bar{\theta}_{i',j'} - \bar{\theta}_{i,j})| \\
		&< \varepsilon \max\{|\alpha^1 - \alpha^2|, |r^1_{i',j'} - r^2_{i',j'}|, |(\psi^1_{i',j'} - \psi^1_{i,j}) - (\psi^2_{i,j} - \psi^2_{i,j})| \} \\
		&\leq \varepsilon \max\{|\alpha^1 - \alpha^2|, \|r^1-r^2\|_\infty, |\psi^1_{i',j'} - \psi^2_{i',j'}| + |\psi^1_{i,j} - \psi^2_{i,j})| \} \\
		&\leq  2\varepsilon \max\{|\alpha^1 - \alpha^2|, \|r^1-r^2\|_\infty, \|\psi^1 - \psi^2\|_\infty\} \\
		& = 2\varepsilon \|(\alpha^1,s^1,\psi^1) - (\alpha^2,s^2,\psi^2)\|_X,  
		\end{split}
	\end{equation} 
	provided that $\max\{|\alpha^1|,|r^1_{i',j'}|,|\psi^1_{i',j'}|, |\psi^1_{i,j}|\}, \max\{|\alpha^2|,|r^2_{i',j'}|,|\psi^2_{i',j'}|, |\psi^2_{i,j}|\} < \delta_1$. Then $(\ref{CosineDiff})$ holds for all $(i,j)\in\Lambda$ when $\|(\alpha^1,s^1,\psi^1)\|_X, \|(\alpha^2,s^2,\psi^2)\|_X < \min\{\delta_1,\frac{a}{2}\}$. Similarly, from condition $(1)$ of Hypothesis $\ref{Hyp:LambdaOmega}$ we have assumed that $\lambda$ is continuously differentiable on $[0,\infty)$, and since $a > 0$ we may apply Theorem $\ref{thm:StrongDiff}$ to the function $f:\mathbb{R}\to\mathbb{R}$ acting by $f(x) = x\lambda(x)$ at the point $x=a$ to obtain a $\delta_2 > 0$ such that
	\begin{equation} \label{lambdaDiff}
		\begin{split}
		|r^1_{i,j}\lambda(r^1_{i,j}) - r^2_{i,j}\lambda(r^2_{i,j}) - a\lambda'(a)(s^1_{i,j} - s^2_{i,j})| &< \varepsilon|s^1_{i,j} - s^2_{i,j}| \\ 
		&\leq \varepsilon \|s^1 - s^2\|_\infty \\
		&\leq \varepsilon \|(\alpha^1,s^1,\psi^1) - (\alpha^2,s^2,\psi^2)\|_X,
		\end{split}
	\end{equation}  
	provided $|r^1_{i,j}|,|r^2_{i,j}| < \delta_2$. This in turn shows that $(\ref{lambdaDiff})$ holds for all $(i,j)\in\Lambda$ when $\|s^1\|_\infty,\|s^2\|_\infty < \min\{\frac{a}{2},\delta_2\}$.

	As a final aside, again from Theorem $\ref{thm:StrongDiff}$ applied to the function $f:\mathbb{R}^2\to\mathbb{R}$ acting by $f(x_1,x_2) = x_1x_2$ at the point $(x_1,x_2) = (0,a)$ there exists a $\delta_3 > 0$ such that 
	\begin{equation} \label{linearDiff}
		\begin{split}
		|\alpha^1r^1_{i,j} - \alpha^2r^2_{i,j} - (\alpha^1 - \alpha^2)a| &< \varepsilon \max\{|\alpha^1 - \alpha^2|, |r^1_{i,j} - r^2_{i,j}|\} \\
		&\leq \varepsilon\max\{|\alpha^1 - \alpha^2|, \|r^1 - r^2\|_\infty\} \\
		&\leq \varepsilon \|(\alpha^1,s^1,\psi^1) - (\alpha^2,s^2,\psi^2)\|_X, 
		\end{split}
	\end{equation}
	provided $ \max\{|\alpha^1 - \alpha^2|, |r^1_{i,j} - r^2_{i,j}|\} < \delta_3$. Then we have that $(\ref{linearDiff})$ holds for all $(i,j)\in \Lambda$ when $\max\{|\alpha^1 - \alpha^2|, \|r^1 - r^2\|_\infty\} < \delta_3$. 
	
Therefore, we take
$\delta = \min\{\frac{a}{2},\delta_1,\delta_2,\delta_3\} > 0$ and assume
$\|(\alpha^1,s^1,\psi^1)\|_X$,\\ $\|(\alpha^2,s^2,\psi^2)\|_X$
$< \delta$. Then this gives
	\begin{equation}
		\begin{split}
			&|[G^1(\alpha^1, r^1,\theta^1) - G^1(\alpha^2, r^2,\theta^2) - M_{21}(\alpha^1 - \alpha^2) - M_{22}(s^1 - s^2)]_{i,j}| \\
			&\leq |\sum_{i',j'} [\alpha^1r^1_{i',j'}\cos(\theta^1_{i',j'} - \theta^1_{i,j}) - \alpha^2r^2_{i',j'}\cos(\theta^1_{i',j'} - \theta^1_{i,j}) - (\alpha^1 - \alpha^2)a\cos(\bar{\theta}_{i',j'} - \bar{\theta}_{i,j})]| \\
			&\ \ \ \ \ + |\sum_{i',j'} [\alpha^1r^1_{i,j} - \alpha^2r^2_{i,j} - (\alpha^1 - \alpha^2)a]| + |r^1_{i,j}\lambda(r^1_{i,j}) - r^2_{i,j}\lambda(r^2_{i,j}) - a\lambda'(a)(s^1_{i,j} - s^2_{i,j})| \\
			&< 2\varepsilon\sum_{i',j'}\|(\alpha^1,s^1,\psi^1) - (\alpha^2,s^2,\psi^2)\|_X \\ 
			&\ \ \ \ \ + \varepsilon\sum_{i',j'} \|(\alpha^1,s^1,\psi^1) - (\alpha^2,s^2,\psi^2)\|_X + \varepsilon \|(\alpha^1,s^1,\psi^1) - (\alpha^2,s^2,\psi^2)\|_X \\
			&\leq 8\varepsilon\|(\alpha^1,s^1,\psi^1) - (\alpha^2,s^2,\psi^2)\|_X + 4\varepsilon\|(\alpha^1,s^1,\psi^1) - (\alpha^2,s^2,\psi^2)\|_X \\ 
			&\ \ \ \ \ + \varepsilon\|(\alpha^1,s^1,\psi^1) - (\alpha^2,s^2,\psi^2)\|_X \\
			&= 13\varepsilon\|(\alpha^1,s^1,\psi^1) - (\alpha^2,s^2,\psi^2)\|_X,  
		\end{split}
	\end{equation}  
	where we have used the fact that each element has at most four nearest neighbours. Therefore taking the supremum over all $(i,j)\in\Lambda$ shows that 
	\begin{equation}
		\frac{\|G^1(\alpha^1, r^1,\theta^1) - G^1(\alpha^2, r^2,\theta^2) - M_{21}(\alpha^1 - \alpha^2) - M_{22}(s^1 - s^2)\|_\infty}{\|(\alpha^1,s^1,\psi^1) - (\alpha^2,s^2,\psi^2)\|_X} < 13\varepsilon.
	\end{equation}
	Since $\varepsilon > 0$ was arbitrary, this quotient can be made arbitrarily small, showing that $G^1$ is strongly Fr\'echet differentiable at $(\alpha,s,\psi) = (0,0,0)$. 
	
	 Now to show that this Fr\'echet derivative is bounded, first we see that for any $\alpha \in \mathbb{R}$ and $(i,j)\in\Lambda$ we have
	 \begin{equation}
	 	|[M_{21}\alpha]_{i,j}| = |\alpha|\cdot|a\sum_{i',j'} [\cos(\bar{\theta}_{i',j'} - \bar{\theta}_{i,j}) - 1]| \leq 8a|\alpha|, 	
	 \end{equation} 
	 where we have again used the fact that each element has at most four nearest-neighbours. Similarly, for $M_{22}$ we have that
	 \begin{equation}
	 	|[M_{22}s]_{i,j}| = a|\lambda'(a)\|s_{i,j}| \leq a|\lambda'(a)|\cdot \|s\|_\infty
	 \end{equation}
	 for any $s \in \ell^\infty(\Lambda)$ and $(i,j)\in\Lambda$. Putting this all together gives that for any $x = (\alpha,s,\psi) \in X$ we have
	 \begin{equation}
	 	|[[M_{21}\alpha + M_{22}s]_{i,j}]| \leq 8a|\alpha| + a|\lambda'(a)|\cdot\|s\|_\infty \leq (8a + a|\lambda'(a)|)\|x\|_X. 
	 \end{equation}
	 Taking the supremum over $(i,j)\in\Lambda$ gives
	 \begin{equation}
	 	\|M_{21}\alpha + M_{22}s\|_X \leq a(8 + |\lambda'(a)|)\|x\|_X, 	
	 \end{equation}
	 showing that the Fr\'echet derivative is bounded. 
\end{proof} %End of Proof

\begin{lem} \label{lem:G2Derivative} %Derivative of G2
	$G^2(\alpha,s,\psi) : \mathbb{R} \times B_{\frac{a}{2}}(0) \times c_0(\Lambda) \to c_0(\Lambda)$ is strongly Fr\'echet differentiable at $(\alpha,s,\psi) = (0,0,0)$. The Fr\'echet derivative at this point is the bounded linear operator which acts by $(\alpha,s,\psi) \mapsto M_{32}\alpha + M_{33}s$, where $M_{32}$ and $M_{33}$ are as defined in $(\ref{M32})$ and $(\ref{M33})$, respectively.
\end{lem}

\begin{proof}
	As in the proof of Lemma $\ref{lem:G1Derivative}$, we will use the variables $r_{i,j} = a + s_{i,j}$ and $\theta_{i,j} = \bar{\theta}_{i,j} + \psi_{i,j}$ for all $(i,j) \in \Lambda$. We will write $r = \{r_{i,j}\}_{(i,j)\in\Lambda}$ and $\theta = \{\theta_{i,j}\}_{(i,j)\in\Lambda}$ and note that $G^2(\alpha,s,\psi)$ strongly Fr\'echet differentiable at $(\alpha, s,\psi) = (0,0,0)$ is equivalent to $G^2(\alpha,r,\theta)$ strongly Fr\'echet differentiable at $(\alpha,r,\theta) = (0, {\bf a}, \bar{\theta})$.  

Begin by taking any $\varepsilon > 0$ and consider $(\alpha^1,s^1,\psi^1),(\alpha^2,s^2,\psi^2) \in \mathbb{R} \times B_{\frac{a}{2}}(0) \times c_0(\Lambda)$. Here $r^1$ and $r^2$ correspond to $s^1$ and $s^2$, respectively. Similarly $\theta^1$ and $\theta^2$ correspond to $\psi^1$ and $\psi ^2$. 

Using the fact that sine is infinitely differentiable on $\mathbb{R}$ and that $a > 0$ to apply Theorem $\ref{thm:StrongDiff}$ and find that there exists a $\delta_1 > 0$ such that
\begin{equation} \label{SineDiff1}
	\begin{split}
		&\bigg|\frac{r^1_{i',j'}}{r^1_{i,j}}\sin(\theta^1_{i',j'} - \theta^1_{i,j}) - \frac{r^2_{i',j'}}{r^2_{i,j}}\sin(\theta^2_{i',j'} - \theta^2_{i,j}) -\frac{1}{a}(s^1_{i',j'} - s^2_{i',j'})\sin(\bar{\theta}_{i',j'} - \bar{\theta}_{i,j}) \\
		& \ \ \ \ \ -\cos(\bar{\theta}_{i',j'} - \bar{\theta}_{i,j})[(\psi^1_{i',j'} - \psi^1_{i,j}) - (\psi^2_{i',j'} - \psi^2_{i,j})] + \frac{1}{a}(s^1_{i,j} - s^2_{i,j})\sin(\bar{\theta}_{i',j'} - \bar{\theta}_{i,j}) \bigg| \\
		&< \varepsilon \max\{|s^1_{i',j'} - s^2_{i',j'}|,|s^1_{i,j} - s^2_{i,j}|,|(\psi^1_{i',j'} - \psi^1_{i,j}) - (\psi^2_{i',j'} - \psi^2_{i,j})|\} \\
		&\leq \varepsilon \max\{\|s^1 - s^2\|_\infty, |\psi^1_{i',j'} - \psi^2_{i',j'}| + |\psi^1_{i,j} - \psi^2_{i,j}|\} \\
		&\leq 2\varepsilon \max\{\|s^1 - s^2\|_\infty, \|\psi^1 - \psi^2\|_\infty\} \\
		&\leq 2\varepsilon \|(\alpha^1,s^1,\psi^1) - (\alpha^2,s^2,\psi^2)\|_X, 
	\end{split}
\end{equation}	
provided $|s^1_{i',j'}|$, $|s^2_{i',j'}|$, $|s^1_{i'j}|$, $|s^2_{i,j}|$, $|\psi^1_{i',j'}|$, $|\psi^2_{i',j'}|$, $|\psi^1_{i,j}|$, $|\psi^2_{i,j}|$ $< \delta_1$. Therefore, $(\ref{SineDiff1})$ holds for every $(i,j)\in\Lambda$ whenever $\|s^1\|_\infty$, $\|s^2\|_\infty$, $\|\psi^1\|_\infty$, $\|\psi^2\|_\infty$ $< \min\{\frac{a}{2},\delta_1\}$.

Now, recalling that 
\begin{equation}
	\sum_{i',j'} \sin(\bar{\theta}_{i',j'} - \bar{\theta}_{i,j}) = 0
\end{equation}
for every $(i,j)\in\Lambda$, we trivially have that 
\begin{equation} \label{SineZero}
	\frac{1}{a}(s^1_{i,j} - s^2_{i,j})\sum_{i',j'} \sin(\bar{\theta}_{i',j'} - \bar{\theta}_{i,j}) = 0	
\end{equation}
for every $(i,j)\in\Lambda$. Therefore, combining this result with $(\ref{SineDiff1})$ above gives that whenever $\|s^1\|_\infty$, $\|s^2\|_\infty$, $\|\psi^1\|_\infty$, $\|\psi^2\|_\infty$ $< \min\{\frac{a}{2},\delta_1\}$ we have
\begin{equation} \label{SineDiff2}
	\begin{split}
		&\bigg|\sum_{i',j'} \frac{r^1_{i',j'}}{r^1_{i,j}}\sin(\theta^1_{i',j'} - \theta^1_{i,j}) - \sum_{i',j'} \frac{r^2_{i',j'}}{r^2_{i,j}}\sin(\theta^2_{i',j'} - \theta^2_{i,j}) -\frac{1}{a}\sum_{i',j'} (s^1_{i',j'} - s^2_{i',j'})\sin(\bar{\theta}_{i',j'} - \bar{\theta}_{i,j}) \\
		& \ \ \ \ \  -\sum_{i',j'} \cos(\bar{\theta}_{i',j'} - \bar{\theta}_{i,j})[(\psi^1_{i',j'} - \psi^1_{i,j}) - (\psi^2_{i',j'} - \psi^2_{i,j})] \bigg| \\
		&\stackrel{(\ref{SineZero})}{\leq} \sum_{i',j'}|\frac{r^1_{i',j'}}{r^1_{i,j}}\sin(\theta^1_{i',j'} - \theta^1_{i,j}) - \frac{r^2_{i',j'}}{r^2_{i,j}}\sin(\theta^2_{i',j'} - \theta^2_{i,j}) -\frac{1}{a}(s^1_{i',j'} - s^2_{i',j'})\sin(\bar{\theta}_{i',j'} - \bar{\theta}_{i,j}) \\
		& \ \ \ \ \ -\cos(\bar{\theta}_{i',j'} - \bar{\theta}_{i,j})[(\psi^1_{i',j'} - \psi^1_{i,j}) - (\psi^2_{i',j'} - \psi^2_{i,j})] + \frac{1}{a}(s^1_{i,j} - s^2_{i,j})\sin(\bar{\theta}_{i',j'} - \bar{\theta}_{i,j}) \bigg| \\
		&\stackrel{(\ref{SineDiff1})}{<} 2\varepsilon \sum_{i',j'}  \|(\alpha^1,s^1,\psi^1) - (\alpha^2,s^2,\psi^2)\|_X \\
		&\leq 8\varepsilon  \|(\alpha^1,s^1,\psi^1) - (\alpha^2,s^2,\psi^2)\|_X,   
	\end{split}
\end{equation}	
where we have used the fact that each element has at most four nearest-neighbours. 

In a similar fashion, we recall that from condition $(2)$ of Hypothesis $\ref{Hyp:LambdaOmega}$ we have that $\omega_1$ is continuously differentiable in both its arguments on $[0,\infty)\times\mathbb{R}$. Let us denote 
\begin{equation}
	K:= \partial_1 \omega_1(a + \rho,\alpha)\bigg|_{(\rho,\alpha) = (0,0)},
\end{equation}
where $\partial_k$ denotes the partial derivative with respect to the $k$th argument. One should also note that since $\omega_1(a,\alpha) = 0$ for all $\alpha > 0$, we have that
\begin{equation}
	\partial_2 \omega_1(a + \rho,\alpha)\bigg|_{(\rho,\alpha) = (0,0)} = 0.	
\end{equation} 
Now since $a>0$, Theorem $\ref{thm:StrongDiff}$ gives that there exists a $\delta_2 > 0$ such that 
\begin{equation} \label{omegaDiff}
	\begin{split}
		|\omega_1(r^1_{i,j},\alpha^1) - \omega_1(r^2_{i,j},\alpha) - K(s^1_{i,j} - s^2_{i,j})| &< \varepsilon\max\{|\alpha^1 - \alpha^2|,|r^1_{i,j} - r^2_{i,j}|\} \\
		&\leq \varepsilon \max\{|\alpha^1 - \alpha^2|,\|r^1 - r^2\|_\infty\} \\
		&\leq \varepsilon \|(\alpha^1,s^1,\psi^1) - (\alpha^2,s^2,\psi^2)\|_X,
	\end{split}
\end{equation}
for any $|\alpha^1|$, $|\alpha^2|$, $|s^1_{i,j}|$, $|s^2_{i,j}|$ $<\delta_2$. Therefore $(\ref{omegaDiff})$ holds for all $(i,j)\in\Lambda$ provided $|\alpha^1|$, $|\alpha^2|$, $\|s^1\|_\infty$, $\|s^2\|_\infty$ $< \min\{\frac{a}{2},\delta_2\}$.

Let $\delta = \min\{\frac{a}{2},\delta_1,\delta_2\} > 0$. Then for all $\|(\alpha^1,s^1,\psi^1)\|_X$, $\|(\alpha^2,s^2,\psi^2)\|_X$ $< \delta$ we get that
\begin{equation}
	\begin{split}
		&|[G^2(\alpha^1, r^1,\theta^1) - G^2(\alpha^2, r^2,\theta^2) - M_{32}(s^1 - s^2) - M_{22}(\psi^1 - \psi^2)]_{i,j}| \\
		&\leq i^{-1}\bigg|\sum_{i',j'} \frac{r^1_{i',j'}}{r^1_{i,j}}\sin(\theta^1_{i',j'} - \theta^1_{i,j}) - \sum_{i',j'} \frac{r^2_{i',j'}}{r^2_{i,j}}\sin(\theta^2_{i',j'} - \theta^2_{i,j}) \\
		&\ \ \ \ \ -\frac{1}{a}\sum_{i',j'} (s^1_{i',j'} - s^2_{i',j'})\sin(\bar{\theta}_{i',j'}- \bar{\theta}_{i,j}) \\
		& \ \ \ \ \   -\sum_{i',j'} \cos(\bar{\theta}_{i',j'} - \bar{\theta}_{i,j})[(\psi^1_{i',j'} - \psi^1_{i,j}) - (\psi^2_{i',j'} - \psi^2_{i,j})] \bigg|  \\
		& \ \ \ \ \ + i^{-1}|\omega_1(r^1_{i,j},\alpha^1) - \omega_1(r^2_{i,j},\alpha) - K(s^1_{i,j} - s^2_{i,j})| \\
		&\stackrel{(\ref{SineDiff2}),(\ref{omegaDiff})}{<} 8i^{-1}\varepsilon \|(\alpha^1,s^1,\psi^1) - (\alpha^2,s^2,\psi^2)\|_X + i^{-1}\varepsilon \|(\alpha^1,s^1,\psi^1) - (\alpha^2,s^2,\psi^2)\|_X \\
		& = 9\|(\alpha^1,s^1,\psi^1) - (\alpha^2,s^2,\psi^2)\|_X,
	\end{split}
\end{equation}
since $i \geq 1$. Therefore, taking the supremum over $(i,j)\in\Lambda$ and dividing by $\|(\alpha^1,s^1,\psi^1) - (\alpha^2,s^2,\psi^2)\|_X$ gives the quotient
\begin{equation}
	\frac{\|G^2(\alpha^1, r^1,\theta^1) - G^2(\alpha^2, r^2,\theta^2) - M_{32}(s^1 - s^2) - M_{22}(\psi^1 - \psi^2)\|_\infty}{\|(\alpha^1,s^1,\psi^1) - (\alpha^2,s^2,\psi^2)\|_X} < 9\varepsilon.	
\end{equation}
Since $\varepsilon > 0$ was arbitrary, this quotient can be made as small as necessary, thus showing that $G^2$ is strongly Fr\'echet differentiable at $(\alpha,s,\psi) = (0,0,0)$.

To see that this Fr\'echet derivative is bounded, we first observe that for any $s\in c_0(\Lambda)$ and $(i,j)\in\Lambda$ we have
\begin{equation}
	|[M_{32}s]_{i,j}| \leq \frac{1}{a}i^{-1}\sum_{i',j'} |s_{i',j'}\|\sin(\bar{\theta}_{i',j'}- \bar{\theta}_{i,j})| \leq \frac{4}{a} \|s\|_\infty,
\end{equation}
where we have used the fact that each element has at most four nearest-neighbours. Similarly, for any $\psi\in c_0(\Lambda)$ and $(i,j)\in\Lambda$ we get
\begin{equation}
	|[M_{33}\psi]_{i,j}| \leq i^{-1}\sum_{i',j'} |\cos(\bar{\theta}_{i',j'}- \bar{\theta}_{i,j})\|\psi_{i',j'} - \psi_{i,j}| \leq 2\sum_{i',j'} \|\psi\|_\infty \leq 8\|\psi\|_\infty.
\end{equation}	 
Therefore, for $(\alpha,s,\psi) \in X$ we have
\begin{equation}
	|[M_{32}s + M_{33}\psi]_{i,j}| \leq \frac{4}{a} \|s\|_\infty + 8\|\psi\|_\infty \leq \bigg(\frac{4}{a} + 8\bigg)\|x\|_X.  	
\end{equation}
Taking the supremum over $(i,j)\in\Lambda$ gives that this Fr\'echet derivative is bounded. 
\end{proof}

We can now prove Proposition $\ref{prop:FrechetDerivative}$.

\begin{proof}[Proof of Proposition \ref{prop:FrechetDerivative}]
	Let us fix $\varepsilon > 0$. Then to begin let us note that for any $(\alpha^1,s^1,\psi^1),(\alpha^2,s^2,\psi^2)\in X$ and $M$ as defined in $(\ref{DerivativeMatrix})$ we have
	\begin{equation}
		\begin{split}
			&\|G(\alpha^1,s^1,\psi^1) - G(\alpha^2,s^2,\psi^2) - M[(\alpha^1,s^1,\psi^1)-(\alpha^2,s^2,\psi^2)]\|_X \\
			&= \max\{0, \|G^1(\alpha^1,s^1,\psi^1) - G^1(\alpha^2,s^2,\psi^2) - M_{21}(\alpha^1 - \alpha^2) - M_{22}(s^1 - s^2)\|_\infty, \\
			& \ \ \ \ \ \|G^2(\alpha^1,s^1,\psi^1) - G^2(\alpha^2,s^2,\psi^2) - M_{32}(s^1 - s^2) - M_{22}(\psi^1 - \psi^2)\|_\infty\}.
		\end{split}
	\end{equation}
	
	Now from Lemma $\ref{lem:G1Derivative}$, we have that there exists $\delta_1 > 0$ such that
	\begin{equation}
		\begin{split}
		\|G^1(\alpha^1,s^1,\psi^1) - &G^1(\alpha^2,s^2,\psi^2) - M_{21}(\alpha^1 - \alpha^2) - M_{22}(s^1 - s^2)\|_\infty \\ 
		& \ \ \ \ \ < \varepsilon \|(\alpha^1,s^1,\psi^1) - (\alpha^2,s^2,\psi^2)\|_X	
		\end{split}
	\end{equation}
	for any elements of $X$ satisfying $\|(\alpha^1,s^1,\psi^1)\|, \|(\alpha^2,s^2,\psi^2)\| < \min\{\frac{a}{2},\delta_1\}$. Similarly, from Lemma $\ref{lem:G2Derivative}$ we have that there exists a $\delta_2 > 0$ such that
	\begin{equation}
		\begin{split}
		\|G^2(\alpha^1,s^1,\psi^1) - &G^2(\alpha^2,s^2,\psi^2) - M_{32}(s^1 - s^2) - M_{33}(\psi^1 - \psi^2)\|_\infty \\ 
		& \ \ \ \ \ < \varepsilon \|(\alpha^1,s^1,\psi^1) - (\alpha^2,s^2,\psi^2)\|_X	
		\end{split}
	\end{equation}
	for any elements of $X$ satisfying $\|(\alpha^1,s^1,\psi^1)\|, \|(\alpha^2,s^2,\psi^2)\| < \min\{\frac{a}{2}, \delta_2\}$. Therefore, any pair of elements of $X$ satisfying $\|(\alpha^1,s^1,\psi^1)\|, \|(\alpha^2,s^2,\psi^2)\| < \min\{\frac{a}{2}, \delta_1,\delta_2\}$ we get  
	\begin{equation}
		\begin{split}
			&\|G(\alpha^1,s^1,\psi^1) - G(\alpha^2,s^2,\psi^2) - M[(\alpha^1,s^1,\psi^1)-(\alpha^2,s^2,\psi^2)]\|_X \\
			&< \varepsilon \max\{0,  \|(\alpha^1,s^1,\psi^1) - (\alpha^2,s^2,\psi^2)\|_X\} \\
			&\leq \varepsilon  \|(\alpha^1,s^1,\psi^1) - (\alpha^2,s^2,\psi^2)\|_X.
		\end{split}
	\end{equation}
	Since $\varepsilon > 0$ was arbitrary, we have shown that $G$ is strongly Fr\'echet differentiable at $(\alpha,s,\psi) = (0,0,0)$. Boundedness of $M:X\to X$ is a trivial consequence of Lemmas $\ref{lem:G1Derivative}$ and $\ref{lem:G2Derivative}$, thus completing the proof.  
\end{proof}

%Bibliography ----------------------------------------------------------------------------------------------------------------------------------------------------------------------------------------------------------------------------------------------

\end{document}